\documentclass[preprint,12pt]{elsarticle}

\usepackage{amsmath,amssymb,amsthm,mathtools,mathrsfs,bm}
\usepackage{enumitem}
\usepackage{hyperref}

\biboptions{numbers,sort&compress}

\hypersetup{colorlinks=true,linkcolor=blue,citecolor=blue,urlcolor=blue}

\theoremstyle{plain}
\newtheorem{theorem}{Theorem}[section]
\newtheorem{proposition}[theorem]{Proposition}
\newtheorem{lemma}[theorem]{Lemma}

\theoremstyle{definition}

\newtheorem{assumption}[theorem]{Assumption}
\newtheorem{remark}[theorem]{Remark}

\newcommand{\one}{\mathbf{1}}
\newcommand{\Lip}{\operatorname{Lip}}
\newcommand{\supp}{\operatorname{supp}}
\newcommand{\Law}{\mathcal{L}}
\newcommand{\Risk}{\mathcal{R}}
\newcommand{\EmpRisk}{\widehat{\mathcal{R}}}
\newcommand{\ClassF}{\mathcal{F}}
\newcommand{\ClassG}{\mathcal{G}}
\newcommand{\ClassB}{\mathcal{B}}
\newcommand{\ClassQ}{\mathcal{Q}}
\newcommand{\NN}{\mathcal{N}}
\newcommand{\norminf}[1]{\lVert #1 \rVert_\infty}
\newcommand{\norm}[1]{\lVert #1 \rVert}
\newcommand{\abs}[1]{\lvert #1 \rvert}
\newcommand{\bbR}{\mathbb{R}}
\newcommand{\bbE}{\mathbb{E}}
\newcommand{\bbP}{\mathbb{P}}
\newcommand{\bbN}{\mathbb{N}}
\newcommand{\ClassGb}{\overline{\mathcal{G}}}

\begin{document}

\begin{frontmatter}

\title{Drift Estimation for a Multi-Dimensional L\'evy-Driven
Stochastic Differential Equation Using Deep Neural Networks}

\author{Taisei Noguchi}
\author{Teppei Ogihara}
\address{Graduate School of Information Science and Technology,
The University of Tokyo, Tokyo, Japan}
\begin{abstract}
We develop a non-asymptotic theory for nonparametric drift estimation in discretely observed multi-dimensional L\'evy-driven stochastic differential equations using sparse ReLU neural networks. Under exponential \(\beta\)-mixing and an exponential-tail condition on the L\'evy measure, we establish an oracle inequality for the least-squares estimator. The jump component changes both the martingale structure and the concentration regime relative to diffusion models: continuous martingales are replaced by discontinuous martingales, while the relevant empirical-process fluctuations are sub-exponential rather than sub-Gaussian. We handle these difficulties without truncating the observed increments, using an exponential-supermartingale argument for compensated Poisson integrals and \(\psi_1\)-chaining. Despite the weaker concentration, the resulting statistical bound is, up to constants, of the same order as the corresponding diffusion bound. For drift functions with hierarchical compositional structure, this yields intrinsic-dimensional convergence rates. The framework allows infinite jump activity and, in some cases, infinite variation. Finally, we establish a minimax lower bound on a fixed non-trivial compound-Poisson submodel. Together with the upper bound, this shows that the estimator is minimax optimal up to logarithmic factors.
\end{abstract}


\end{frontmatter}

\section{Introduction}

Let us consider a \(d\)-dimensional c\`adl\`ag process
\(X=(X_t)_{t\ge0}\) satisfying the following L\'evy-driven stochastic
differential equation:
\begin{equation}\label{eq:model}
\begin{aligned}
X_t
={}&
\eta
+\int_0^t b(X_{s-})\,ds
+\int_0^t \Sigma(X_{s-})\,dW_s  \\
&\quad
+\int_0^t\int_{\bbR^d\setminus\{0\}}
\Xi(X_{s-})z\,\widetilde N(ds,dz),
\qquad t\ge0.
\end{aligned}
\end{equation}
Here
\(b:\bbR^d\to\bbR^d\) is the drift function,
\(\Sigma:\bbR^d\to\bbR^{d\times d}\) is the diffusion coefficient,
and \(\Xi:\bbR^d\to\bbR^{d\times d}\) is the jump coefficient.
Moreover, \(\eta\) is a \(d\)-dimensional random vector,
\(W=(W_t)_{t\ge0}\) is a \(d\)-dimensional standard Wiener process,
and \(N\) is a Poisson random measure on
\(
[0,\infty)\times\bigl(\bbR^d\setminus\{0\}\bigr)
\)
with intensity measure \(ds\,\nu(dz)\), where \(\nu\) is a L\'evy
measure on \(\bbR^d\setminus\{0\}\). We write
\(
\widetilde N(ds,dz):=N(ds,dz)-ds\,\nu(dz)
\)
for the associated compensated Poisson random measure, and assume that
\(W\), \(N\), and \(\eta\) are mutually independent. Our aim is to
estimate the drift function \(b\) nonparametrically from the discrete
observations of \(X\).

We observe a single trajectory of \(X\) at the equidistant times
\(
    0,\Delta,\ldots,n\Delta,
\)
and primarily have in mind the high-frequency and long-term regime
\(
    \Delta=\Delta_n\to0,
    \qquad
    n\Delta\to\infty.
\)
For a fixed \(i\in\{1,\ldots,d\}\), we estimate the \(i\)-th component
of the drift on a compact region of the state space.
After an affine rescaling, we take this region to be \([0,1]^d\) and set
\(
    f_0:=b^i\one_{[0,1]^d}.
\)
Since the components can be estimated separately, the resulting
estimators can be combined to recover the vector-valued drift \(b\).
Although the above asymptotic regime motivates our analysis, our principal
oracle bounds are non-asymptotic.

Nonparametric drift estimation for diffusion processes has been studied
extensively. In the one-dimensional setting, various estimation methods have
been developed under continuous, high-frequency discrete, and low-frequency
discrete observation schemes; see, for example,
Dalalyan~\cite{dal05},
Comte et al.~\cite{com-etal07}, and
Gobet et al.~\cite{gob-etal04}.
The theory has also been extended to multi-dimensional diffusions,
particularly when a continuous record of the process is available; see
Dalalyan and Rei{\ss}~\cite{dal-rei07} and
Nickl and Ray~\cite{nic-ray20}.
More recently, Oga and Koike~\cite{oga-koi24} introduced deep neural
networks into nonparametric drift estimation for discretely observed
multi-dimensional diffusions; we discuss their results in more detail below.

In contrast, despite the importance of L\'evy-driven and jump-diffusion models in finance and other applied fields, the literature on nonparametric drift estimation for such models is relatively limited.
For a discretely observed one-dimensional jump diffusion,
Schmisser~\cite{sch14} constructed adaptive penalized
least-squares estimators based on linear approximation spaces.
The paper studied both a non-truncated estimator under a sub-exponential
tail condition and a truncated estimator designed to accommodate more
general small-jump activity.
Mies~\cite{mie21} considered one-dimensional Markovian semimartingales
and developed pointwise estimators for the drift and state-dependent jump
activity, using bounded transformations and smooth jump filtering to handle,
under symmetry conditions, possibly infinite-variation and
infinite-variance jumps.
For multi-dimensional jump diffusions, Dexheimer~\cite{dex25}
constructed adaptive kernel estimators from continuous observations and
derived convergence rates under the supremum-norm risk.
Neural-network methods have also recently been explored for learning the
coefficients of L\'evy-driven stochastic systems.
Fang et al.~\cite{fan-etal22} developed a deep-learning approach for
systems driven by \(\alpha\)-stable L\'evy noise, while
Ram\'irez-Gonz\'alez and Sun~\cite{ram-sun26} considered neural-network
estimation for SDEs with finite-activity symmetric jumps.
These works are primarily methodological and numerical and do not provide
non-asymptotic statistical risk bounds or convergence-rate theory for the
resulting neural-network drift estimators.

In parallel, a large body of work has established theoretical advantages of
deep neural networks in nonparametric regression.
Schmidt-Hieber~\cite{sch-hie20} showed that sparse ReLU network
estimators can exploit a compositional structure of the regression function
and attain the corresponding minimax rate up to logarithmic factors.
Kohler and Langer~\cite{koh-lan21} obtained related convergence
rates for fully connected deep neural networks under hierarchical
composition models.
Related theoretical results further highlight the advantages of deep neural networks over conventional linear estimators. Suzuki~\cite{suz19} established the adaptivity of deep ReLU networks for Besov and mixed-smoothness classes, while Imaizumi and Fukumizu~\cite{ima-fuk19} showed that deep neural networks can effectively estimate piecewise smooth functions with non-smooth boundaries. These results show that deep neural networks can exploit low-dimensional or spatially inhomogeneous structures that are not efficiently captured by standard linear methods. Their analyses, however, are primarily developed
for independent regression data.

Oga and Koike~\cite{oga-koi24} connected this DNN regression theory with
nonparametric inference for stochastic differential equations.
For discretely observed multi-dimensional diffusion processes, they
established an oracle inequality for DNN least-squares estimators and showed
that, under a compositional assumption on the drift, the estimators attain
the corresponding minimax rate up to logarithmic factors.
The proof, however, relies on a substantially different noise argument from
that of Schmidt-Hieber~\cite{sch-hie20}: while the latter exploits the
conditional Gaussian structure of the regression noise, the diffusion
setting is handled through a self-normalized inequality for continuous
local martingales.

This approach does not extend directly to L\'evy-driven models.
As pointed out by Oga and Koike~\cite[Remark~2.6]{oga-koi24}, a self-normalized
extension to discontinuous martingales introduces the realized squared-jump
term \(\sum_{s\le T}(\Delta M_s)^2\), for which the pathwise control used
in the diffusion proof is generally unavailable.
Thus, neither the Gaussian argument underlying the regression theory nor
the continuous-martingale argument developed for diffusions directly
handles the sub-exponential jump fluctuations considered here.
The L\'evy-driven extension was consequently left open, with truncation
suggested as a possible route under sufficiently light-tailed jumps.

We resolve this difficulty without truncating the observed increments or
modifying the least-squares criterion.
Instead of extending the self-normalized approach and controlling the
resulting squared-jump term, we construct an exponential-supermartingale
bound directly for the compensated Poisson integral.
The jump component further requires a different empirical-process control:
its fluctuations are sub-exponential rather than sub-Gaussian, and we
control the corresponding local oscillations over the network class by a
\(\psi_1\)-chaining argument. Viewed in this way, our analysis indicates that the structural adaptivity of deep neural networks highlighted by Schmidt-Hieber~\cite{sch-hie20} is not confined to independent regression, but can also be retained in a dependent stochastic-process setting with sub-exponential jump fluctuations.

These arguments yield a non-asymptotic oracle inequality for sparse ReLU
estimators in the multi-dimensional L\'evy-driven model under a
sub-exponential tail condition.
Despite the weaker concentration induced by the jump component, the
resulting statistical bound has, up to constants, the same order as the
diffusion bound in~\cite{oga-koi24}.
Under the compositional structure of Schmidt-Hieber~\cite{sch-hie20}, this
gives the corresponding DNN convergence rate without deterioration of the
diffusion-case rate.
The framework is not restricted to finite-activity jump diffusions and
allows L\'evy measures with infinite jump activity and, in particular
cases, infinite variation.

Finally, we establish a matching minimax lower bound on a non-trivial
Gaussian compound-Poisson submodel and verify uniform exponential
\(\beta\)-mixing for this submodel.
Together with the upper bound, this shows that the DNN-based least-squares
estimator remains minimax optimal up to logarithmic factors in a genuinely
jump-driven statistical experiment.

The remainder of this paper is organized as follows.
Section~\ref{sec:assumptions} introduces the assumptions and statistical
framework, and Section~\ref{sec:main-results} presents the main results.
Their proofs are given in the subsequent sections, while
\ref{app:psi1-chaining} provides an auxiliary
\(\psi_1\)-chaining inequality used in the analysis of the jump component.

\paragraph{Notation}
For a vector or matrix \(W\), we write \(\abs{W}\) for the Frobenius norm
(i.e. the Euclidean norm for a vector), \(\abs{W}_\infty\) for the
maximum-entry norm, and \(\abs{W}_0\) for the number of non-zero entries.
For vectors \(x,y\in\bbR^d\), \(\langle x,y\rangle:=x^\top y\) denotes
the inner product. We write
\(\norminf{f}:=\sup_{x\in[0,1]^d}\abs{f(x)}\)
for the supremum norm on \([0,1]^d\), \(\lfloor x\rfloor\) for the largest
integer not exceeding \(x\), and \(\lceil x\rceil\) for the smallest integer
not smaller than \(x\). For a random vector \(\xi\), \(\Law_\xi\) denotes
the law of \(\xi\). We write \(\one_A\) for the indicator function of a set
\(A\). For a c\`adl\`ag process \(X\), \(X_{t-}\) denotes the left limit at
time \(t\).

\section{Assumptions and definitions}
\label{sec:assumptions}

\subsection{Standing assumptions}

We impose the following assumptions.

\begin{assumption}\label{asm:coefficients}
The coefficient functions \(b\), \(\Sigma\), and \(\Xi\) are globally Lipschitz
continuous. Namely, there exist constants
\(C_b', C_\Sigma', C_\Xi'>0\) such that
\[
\abs{b(x)-b(y)} \le C_b'\abs{x-y},\;
\abs{\Sigma(x)-\Sigma(y)} \le C_\Sigma'\abs{x-y},\;
\abs{\Xi(x)-\Xi(y)} \le C_\Xi'\abs{x-y}
\]
for all \(x,y\in\bbR^d\). In addition, the jump coefficient \(\Xi\) is
uniformly bounded on \(\bbR^d\). We write
\[
C_b := \sup_{x\in[0,1]^d}\abs{b(x)}<\infty,
\;
C_\Sigma := \sup_{x\in[0,1]^d}\abs{\Sigma(x)}<\infty,
\;
C_\Xi := \sup_{x\in\bbR^d}\abs{\Xi(x)}<\infty.
\]
\end{assumption}

The global boundedness of the jump coefficient is also commonly imposed
in statistical analyses of jump diffusions; see, for example,
Schmisser~\cite{sch14} and Dexheimer~\cite{dex25}.

Let \((P_t)_{t\ge0}\) be the transition semigroup of \(X\) and set
\(\mu_0:=\Law_\eta\). Following Masuda~\cite{mas07}, we define the
\(\beta\)-mixing coefficient of \(X\) by
\[
\beta_X(t)
:=
\sup_{s\ge0}
\int_{\bbR^d}
\norm{P_t(x,\cdot)-\mu_0P_{s+t}(\cdot)}_{\mathrm{TV}}\,
\mu_0P_s(dx),
\qquad t\ge0,
\]
where \(\norm{\cdot}_{\mathrm{TV}}\) denotes the total variation norm.

\begin{assumption}\label{asm:mixing}
The process \(X\) is exponentially \(\beta\)-mixing, that is, there exist
constants \(C_\beta,C_\beta'>0\) such that
\[
\beta_X(t) \le C_\beta' e^{-C_\beta t},\qquad t>0.
\]
\end{assumption}
\noindent See Remark~\ref{rem:exponential-beta-mixing} for a set of sufficient
conditions for Assumption~\ref{asm:mixing}.

\begin{assumption}\label{asm:sampling}
The sampling scheme satisfies \(\Delta\le 1\) and \(n\Delta\ge 2\).
\end{assumption}

\begin{assumption}\label{asm:levy-tail}
The L\'evy measure \(\nu\) satisfies the following exponential tail
condition: there exist constants \(C_m,\lambda>0\) such that
\[
\nu(\{\abs{z}>r\})\le C_m e^{-\lambda r},
\qquad r\ge 1.
\]
\end{assumption}

\noindent Assumption~\ref{asm:levy-tail} implies that
\(C_\nu:=\int_{\bbR^d\setminus\{0\}}\abs{z}^2\,\nu(dz)<\infty\)
and, more generally, that
\(\Lambda(a):=\int_{\bbR^d\setminus\{0\}}
\abs{z}^2e^{a\abs{z}}\,\nu(dz)<\infty\)
for every \(a\in[0,\lambda)\); see Remark~\ref{rem:Lambda-finite}.

Assumption~\ref{asm:levy-tail} is satisfied, for example, by compound Poisson processes with exponentially decaying jump-size tails, exponentially tempered \(\alpha\)-stable L\'evy processes with \(\alpha\in(0,2)\), including the CGMY process as a prominent one-dimensional example, and normal inverse Gaussian processes; see Rosi\'{n}ski~\cite{ros07}, Carr et al.~\cite{car-etal02}, and Barndorff-Nielsen~\cite{bar98}, respectively, for their definitions and applications. In particular, the assumption allows for infinite jump activity, and even infinite variation. We also note that the jump component in \eqref{eq:model} is written in centered form; for all the examples above, the first moment is finite, and centering the driving L\'evy process leaves its L\'evy measure unchanged, since the resulting deterministic finite-variation term can be absorbed into the drift coefficient \(b\).

\begin{remark}\label{rem:Lambda-finite}
For every \(a\in[0,\lambda)\), the definition of a L\'evy measure gives
\[
\int_{\{0<\abs{z}\le1\}}
\abs{z}^2e^{a\abs{z}}\,\nu(dz)
\le
e^a
\int_{\{0<\abs{z}\le1\}}
\abs{z}^2\,\nu(dz)
<\infty.
\]
For the large-jump part, setting
\(A_k:=\{z\in\bbR^d:k<\abs{z}\le k+1\}\), we obtain
\[
\int_{\{\abs{z}>1\}}
\abs{z}^2e^{a\abs{z}}\,\nu(dz)
\le
C_m e^a
\sum_{k=1}^\infty
(k+1)^2e^{-(\lambda-a)k}
<\infty.
\]
Hence \(\Lambda(a)<\infty\) for every \(a\in[0,\lambda)\); in particular,
\(C_\nu=\Lambda(0)<\infty\).
\end{remark}

\begin{remark}
\label{rem:exponential-beta-mixing}
Assumption~\ref{asm:mixing} is satisfied if, in addition to
Assumptions~\ref{asm:coefficients} and~\ref{asm:levy-tail},
the following conditions hold:
\begin{enumerate}
\item[(i)]
There exist constants \(0<\lambda_-\le\lambda_+<\infty\)
such that
\[
\lambda_-\abs{y}^2
\le
\abs{\Sigma(x)^\top y}^2
\le
\lambda_+\abs{y}^2
\]
for every \(x,y\in\bbR^d\).

\item[(ii)]
There exist constants \(r>0\) and \(M_0>0\) such that
\[
\langle b(x),x\rangle
\le
-r\abs{x}^2
\]
for every \(x\in\bbR^d\) with \(\abs{x}>M_0\).

\item[(iii)]
The initial variable satisfies \(\bbE[\abs{\eta}]<\infty\).
\end{enumerate}

This follows from Masuda~\cite[Theorem~2.2]{mas07}. Indeed, Masuda's Assumption~2(a)
follows from Menozzi, Pesce and Zhang~\cite[Theorem~1.2]{men-etal21}, whereas
Masuda's Assumption~\(3^\ast\) follows from
Masuda~\cite[Lemma~2.4(ii)]{mas07}, applied with
\(q=1\).

The lower non-degeneracy bound in~(i), which is used to verify
Masuda's Assumption~2(a), may be replaced by a suitable
non-degeneracy condition on the jump mechanism; see
Masuda's Assumption~2(b).
\end{remark}

\subsection{Multilayer neural networks}

We use sparse ReLU networks following the framework of
Schmidt-Hieber~\cite{sch-hie20}, localized to the estimation domain
as in Oga and Koike~\cite{oga-koi24}. Let
\(\sigma(x):=\max\{x,0\}\) for \(x\in\bbR\).
For \(\mathbf{v}=(v_1,\ldots,v_r)^\top\in\bbR^r\), define
\[
    \sigma_{\mathbf{v}}(y)
    :=
    \bigl(
        \sigma(y_1-v_1),\ldots,\sigma(y_r-v_r)
    \bigr)^\top,
    \qquad
    y=(y_1,\ldots,y_r)^\top\in\bbR^r.
\]

A network architecture is specified by the number \(L\in\bbN\) of hidden
layers and a width vector
\(\mathbf{p}=(p_0,\ldots,p_{L+1})\in\bbN^{L+2}\).
A neural network with architecture \((L,\mathbf{p})\) is a function
\(g:\bbR^{p_0}\to\bbR^{p_{L+1}}\) of the form
\begin{equation}\label{eq:nn-form}
    g(x)
    =
    W_L\sigma_{\mathbf{v}_L}
    W_{L-1}\sigma_{\mathbf{v}_{L-1}}
    \cdots
    W_1\sigma_{\mathbf{v}_1}W_0x,
    \qquad x\in\bbR^{p_0},
\end{equation}
where \(W_j\in\bbR^{p_{j+1}\times p_j}\) for \(j=0,\ldots,L\) and
\(\mathbf{v}_j\in\bbR^{p_j}\) for \(j=1,\ldots,L\).
We adopt the convention \(\mathbf{v}_0=0\).
Since we estimate real-valued functions on \(\bbR^d\), we always take
\(p_0=d\) and \(p_{L+1}=1\).

For \(s,F\in(0,\infty]\), define
\begin{align}
\ClassF(L,\mathbf{p},s,F)
:=
\Biggl\{
f=g\one_{[0,1]^d}\;:\;&
g \text{ is of the form \eqref{eq:nn-form}}, \notag\\
&
\max_{0\le j\le L}
\left(
    \abs{W_j}_\infty
    \vee
    \abs{\mathbf{v}_j}_\infty
\right)
\le 1, \notag\\
&
\sum_{j=0}^L
\left(
    \abs{W_j}_0+\abs{\mathbf{v}_j}_0
\right)
\le s,
\quad
\norminf{f}\le F
\Biggr\}.
\label{eq:nn-class}
\end{align}
We write
\(
    \ClassF(L,\mathbf{p},s)
    :=
    \ClassF(L,\mathbf{p},s,\infty).
\)

\subsection{Estimator and risks}

Let
\(
    \mathscr D_n:=\sigma(X_{k\Delta}:k=0,\dots,n).
\)
An estimator is a random function
\(\widehat f_n:\Omega\times\bbR^d\to\bbR\) that is
\(\mathscr D_n\otimes\ClassB(\bbR^d)\)-measurable. We assume that the
candidate class \(\ClassF\) is pointwise measurable and that
\(\widehat f_n\) takes values in \(\ClassF\).

Let \(X^i\) denote the \(i\)-th component of \(X\). For
\(k=0,\dots,n-1\), define
\[
    Y_{k\Delta}
    :=
    \frac{X^i_{(k+1)\Delta}-X^i_{k\Delta}}{\Delta}.
\]
As seen from the regression decomposition \eqref{eq:regression},
\(Y_{k\Delta}\) plays the role of a response variable whose regression
function is approximately \(f_0=b^i\one_{[0,1]^d}\). We therefore learn
\(f_0\) by seeking a function \(f\in\ClassF\) with small empirical
least-squares loss
\[
    \ClassQ_n(f)
    :=
    \frac1n\sum_{k=0}^{n-1}
    \bigl(Y_{k\Delta}-f(X_{k\Delta})\bigr)^2.
\]

Since an exact empirical risk minimizer is generally difficult to compute
for a non-convex neural-network class, we allow \(\widehat f_n\) to be an
arbitrary \(\ClassF\)-valued estimator and define its optimization error by
\[
    \Psi_n^{\ClassF}(\widehat f_n)
    :=
    \bbE\left[
        \ClassQ_n(\widehat f_n)
        -
        \inf_{\widetilde f\in\ClassF}\ClassQ_n(\widetilde f)
    \right].
\]
Thus, \(\Psi_n^{\ClassF}(\widehat f_n)\) measures how closely the learning
procedure approaches an exact empirical risk minimizer. For
\(\ClassF=\ClassF(L,\mathbf p,s,F)\), we simply write
\(
    \Psi_n(\widehat f_n)
    :=
    \Psi_n^{\ClassF(L,\mathbf p,s,F)}(\widehat f_n).
\)

Finally, the statistical accuracy of the learned function is measured by
the generalization risk
\[
    \Risk_n(\widehat f_n,f_0)
    :=
    \bbE\left[
        \frac1n\sum_{k=0}^{n-1}
        \bigl(
            \widehat f_n(X'_{k\Delta})
            -
            f_0(X'_{k\Delta})
        \bigr)^2
    \right],
\]
where \(X'\) is an independent copy of \(X\), independent of the observed sample. Hence, \(\ClassQ_n\) is the training criterion,
\(\Psi_n\) measures the optimization error, and \(\Risk_n\) evaluates the
out-of-sample error of the resulting estimator.

\section{Main results}
\label{sec:main-results}

\subsection{Oracle inequality}
Our first main result is the following non-asymptotic oracle inequality,
which separates the optimization and approximation errors from the
stochastic and discretization errors.

\begin{theorem}\label{thm:oracle}
Suppose that Assumptions~\ref{asm:coefficients}--\ref{asm:levy-tail} hold and that \(\norminf{f_0}\le F\) for some \(F\ge 1\). Let \(\widehat f_n\) be any estimator taking values in the class \(\ClassF(L,\mathbf{p},s,F)\) with \(s\ge 2\). Then, for every \(\tau>1\), there exists a constant \(C_\tau>0\), depending only on
\((d,\allowbreak C_b,\allowbreak C_b',\allowbreak C_\Sigma,\allowbreak C_\Sigma',
\allowbreak C_\Xi,\allowbreak C_\Xi',\allowbreak C_\nu,\allowbreak C_m,\allowbreak\lambda,
\allowbreak C_\beta,\allowbreak C_\beta',\allowbreak\tau)\), such that
\[
\begin{aligned}
\Risk_n(\widehat f_n,f_0)
&\le
\tau\Biggl(
\Psi_n(\widehat f_n)
+
\inf_{f\in\ClassF(L,\mathbf{p},s,F)}\norminf{f-f_0}^2
\Biggr)\\
&\quad+
C_\tau F^2
\Biggl(
\frac{s(L\log s+\log(n\Delta))\log(n\Delta)}{n\Delta}
+\Delta
\Biggr).
\end{aligned}
\]
\end{theorem}

Theorem~\ref{thm:oracle} may be viewed as a L\'evy-driven analogue of
Schmidt-Hieber~\cite[Theorem~2]{sch-hie20}.
Compared with the i.i.d.\ nonparametric regression setting, the relevant
information scale is the total observation horizon \(n\Delta\), rather
than the sample size \(n\). Moreover, the temporal dependence of
\((X_{k\Delta})_{k=0}^{n-1}\) gives rise to the additional
\(\log(n\Delta)\) factor, while the term \(\Delta\) accounts for the
discretization error.

More directly, Theorem~\ref{thm:oracle} provides a L\'evy-driven
extension of Oga and Koike~\cite[Theorem~3.1]{oga-koi24}.
As noted in Oga and Koike~\cite[Remark~3.1(b)]{oga-koi24}, their diffusion
bound is closely related to the projection-based oracle inequality of
Schmisser~\cite[Theorem~1]{sch13}, which concerns the expected
empirical risk. In the one-dimensional jump-diffusion setting,
Schmisser~\cite[Theorem~4]{sch14} further obtained, under a
sub-exponential tail condition, a bound of the same order as in the
corresponding diffusion model. Our result gives a similar conclusion
for sparse DNN estimators in a multidimensional setting: under
Assumption~\ref{asm:levy-tail}, the jump terms introduce no additional
rate term relative to Oga and Koike~\cite[Theorem~3.1]{oga-koi24}.
These results suggest that sub-exponential tails provide a natural
regime in which diffusion-type bounds may persist in the presence of
jumps. For heavier-tailed L\'evy measures, Schmisser
\cite[Theorem~10]{sch14} employed truncation and obtained
additional remainder terms depending on the Blumenthal--Getoor index;
an analogous robust DNN procedure is left for future work.

\subsection{Compositional function class}

Following Schmidt-Hieber~\cite{sch-hie20}, we consider a
hierarchical compositional structure in which each component map depends
on only a small subset of its input coordinates. For a closed rectangle \(D=[\ell,u]^p\) and \(m\in\bbN_0\), 
we write \(f\in C^m(D)\) if there exists \(\widetilde f\in C^m(\bbR^p)\)
such that \(f=\widetilde f|_D\). 
For \(\abs{\alpha}_1\le m\), we define
\(\partial^\alpha f := (\partial^\alpha \widetilde f)|_D\).

For \(\beta>0\), let
\(\underline{\beta}:=\max\{m\in\bbN_0:m<\beta\}
=\lceil\beta\rceil-1\), the largest integer strictly smaller than
\(\beta\). For a multi-index
\(\alpha=(\alpha_1,\dots,\alpha_p)\in\bbN_0^p\), write
\(\abs{\alpha}_1:=\sum_{j=1}^p\alpha_j\) and
\(\partial^\alpha
:=\partial_1^{\alpha_1}\cdots\partial_p^{\alpha_p}\).

For integers \(p\ge r\ge1\), we say that a function
\(f:[\ell,u]^p\to\bbR\) depends on at most \(r\) coordinates if there
exists \(S\subset\{1,\dots,p\}\) with \(\abs{S}\le r\) such that
\(f(x)=f(y)\) whenever \(x_j=y_j\) for every \(j\in S\). For
\(\beta,K>0\) and \(\ell<u\), define
\[
\begin{multlined}
\mathcal C_r^\beta([\ell,u]^p,K)
:=
\Biggl\{
f\in C^{\underline{\beta}}([\ell,u]^p):
f\text{ depends on at most }r\text{ coordinates},\\
\qquad
\sum_{\abs{\alpha}_1<\beta}
\norminf{\partial^\alpha f}
+
\sum_{\abs{\alpha}_1=\underline{\beta}}
\sup_{\substack{x,y\in[\ell,u]^p\\x\ne y}}
\frac{
    \abs{\partial^\alpha f(x)-\partial^\alpha f(y)}
}{
    \abs{x-y}_\infty^{\beta-\underline{\beta}}
}
\le K
\Biggr\}.
\end{multlined}
\]

Let \(q\in\bbN_0\),
\(\mathbf d=(d_0,\dots,d_{q+1})\in\bbN^{q+2}\),
\(\mathbf t=(t_0,\dots,t_q)\in\bbN^{q+1}\), and
\(\bm\beta=(\beta_0,\dots,\beta_q)\in(0,\infty)^{q+1}\), where
\(d_0=d\), \(d_{q+1}=1\), and \(t_i\le d_i\) for
\(i=0,\dots,q\).

We consider functions \(f:[0,1]^d\to\bbR\) admitting a representation
\begin{equation}\label{eq:composition}
    f=g_q\circ g_{q-1}\circ\cdots\circ g_0.
\end{equation}
Here, \(\ell_0=0\), \(u_0=1\), and the intervals
\([\ell_i,u_i]\), \(i=1,\dots,q+1\), are allowed to vary subject to
\(\ell_i<u_i\) and
\(\abs{\ell_i}\vee\abs{u_i}\le K\). We define
\[
\begin{multlined}
\ClassG(q,\mathbf d,\mathbf t,\bm\beta,K)
:=
\Biggl\{
f:[0,1]^d\to\bbR:
f=g_q\circ\cdots\circ g_0,\\
\qquad
g_i=(g_{ik})_{k=1}^{d_{i+1}}:
[\ell_i,u_i]^{d_i}
\to
[\ell_{i+1},u_{i+1}]^{d_{i+1}},\\
\qquad
g_{ik}\in
\mathcal C_{t_i}^{\beta_i}
([\ell_i,u_i]^{d_i},K),\quad
i=0,\dots,q,\quad
k=1,\dots,d_{i+1}
\Biggr\}.
\end{multlined}
\]
Here, \(d_i\) is the dimension of the \(i\)-th intermediate
representation, whereas \(t_i\) is the maximal number of coordinates
on which each component \(g_{ik}\) may depend. The compositional depth
\(q\) need not coincide with the neural-network depth \(L\).

For the drift estimation problem, define the zero-extended class
\[
\begin{aligned}
    \ClassGb(q,\mathbf d,\mathbf t,\bm\beta,K)
    &:=
    \Bigl\{
        f:\bbR^d\to\bbR:
        \supp(f)\subset[0,1]^d,\quad
        \\
    &\qquad
        f|_{[0,1]^d}
        \in\ClassG(q,\mathbf d,\mathbf t,\bm\beta,K)
    \Bigr\}.
\end{aligned}
\]

Finally, set
\[
    \beta_i^*
    :=
    \beta_i\prod_{\ell=i+1}^{q}(\beta_\ell\wedge1),
    \;
    \phi_n
    :=
    \max_{0\le i\le q}
    (n\Delta)^{-\frac{2\beta_i^*}{2\beta_i^*+t_i}},
    \;
    c_L
    :=
    \sum_{i=0}^{q}
    \frac{\beta_i+t_i}{t_i}
    \log_2(4t_i\vee4\beta_i).
\]

\begin{theorem}\label{thm:composition}
Suppose that Assumptions~\ref{asm:coefficients}--\ref{asm:levy-tail} are satisfied. Consider the \(d\)-dimensional L\'evy-driven SDE model \eqref{eq:model} such that
\(f_0 \in \ClassGb(q,\mathbf{d},\mathbf{t},\bm{\beta},K)\). Let
\(\widehat f_n\) be an estimator taking values in the network class \(\ClassF(L,\mathbf{p},s,F)\) and assume:
\begin{enumerate}[label=(\roman*)]
  \item \(F \ge K \vee 1\);
  \item \(11q+8+c_L\log_2(n\Delta\phi_n)\le L \le c_L^u n\Delta\phi_n\);
  \item \(c_p n\Delta\phi_n \le \min_{1\le i\le L} p_i\);
  \item \(c_s^{\ell} n\Delta\phi_n \log(n\Delta)\le s \le c_s^u n\Delta\phi_n \log(n\Delta)\),
\end{enumerate}
for some constants \(c_L^u,c_p,c_s^{\ell},c_s^u>0\). Then there exist constants \(C,C',C''>0\), depending only on
\((C_b,\allowbreak C_b',\allowbreak C_\Sigma,\allowbreak C_\Sigma',\allowbreak C_\Xi,
\allowbreak C_\Xi',\allowbreak C_\beta,\allowbreak C_\beta',\allowbreak C_m,\allowbreak\lambda,
\allowbreak c_L^u,\allowbreak c_p,\allowbreak c_s^{\ell},\allowbreak c_s^u,
\allowbreak q,\allowbreak\mathbf{d},\allowbreak\mathbf{t},\allowbreak\bm{\beta},\allowbreak F)\), such that
\[
\Risk_n(\widehat f_n,f_0)
\le
C\Psi_n(\widehat f_n)
+
C'\phi_n L\log^3(n\Delta)
+
C''\Delta.
\]
\end{theorem}

\subsection{Minimax optimality}

We finally show that the rate in Theorem~\ref{thm:composition} is optimal up
to a logarithmic factor in the minimax sense. In the pure diffusion setting,
Oga and Koike~\cite[Lemma~3.1, Proposition~3.1 and Theorem~3.3]{oga-koi24}
prove a uniform exponential \(\beta\)-mixing bound on a tractable diffusion
submodel, the corresponding uniform upper bound, and the matching minimax
lower bound.

If one only wanted a lower bound for a larger model containing the pure
diffusion case, one could restrict the model to that diffusion submodel. We do
not take this shortcut. Instead, we prove the lower bound on a fixed
non-trivial compound-Poisson submodel. This shows that the rate remains
unimprovable in a genuinely jump-driven statistical experiment, not merely
because the whole model contains the pure diffusion case. Since minimax lower
bounds may be proved on any smaller submodel, it is enough for our purpose to
work with this fixed jump mechanism.

Fix a jump intensity \(\iota>0\) and consider the following
compound-Poisson submodel:
\begin{equation}
\label{eq:gaussian-cp-submodel}
    dX_t=b(X_{t-})\,dt+dW_t+dJ_t,\qquad X_0=\eta .
\end{equation}
Here
\[
    J_t=\sum_{m=1}^{N_t}Z_m,\qquad t\ge0,
\]
where \(N=(N_t)_{t\ge0}\) is a Poisson process with intensity \(\iota\),
and \((Z_m)_{m\ge1}\) are i.i.d.\ \(N(0,I_d)\)-random vectors independent
of \(N\), \(W\), and \(\eta\). Equivalently, the L\'evy measure of \(J\) is
\[
    \nu_J(dz)=\iota\varphi_d(z)\,dz,
    \qquad
    \varphi_d(z):=(2\pi)^{-d/2}
    \exp\left(-\frac{\abs{z}^2}{2}\right).
\]
Throughout this subsection, this jump mechanism is fixed and only the drift
coefficient \(b\) varies.

For \(r\in(0,1]\) and \(K>0\), let \(\ClassB_0^J(d,r,K)\) be the class of
measurable functions \(b:\bbR^d\to\bbR^d\) satisfying
\[
    \langle b(x),x\rangle\le -r\abs{x}
    \quad\text{for all } \abs{x}>2d,
    \qquad
    \sup_{x\in\bbR^d}\abs{b(x)}_\infty\le K .
\]
For every \(b\in\ClassB_0^J(d,r,K)\),
\eqref{eq:gaussian-cp-submodel} admits a unique weak solution.
Indeed, by Jacod and Shiryaev~\cite[Corollary~III.2.41]{jac-shi03},
the associated martingale problem is well posed, since \(b\) is bounded,
the diffusion matrix is the identity, and the Lévy kernel
\(\nu_J(dz)=\iota\varphi_d(z)\,dz\) is finite and independent of the state variable.

The Gaussian specification is made for concreteness; its role is
discussed in Remark~\ref{rem:fixed-cp-extension} below.

The following lemma is the analogue of
Oga and Koike~\cite[Lemma~3.1]{oga-koi24} for the Gaussian compound-Poisson submodel.

\begin{lemma}
\label{lem:uniform-mixing-gaussian-cp}
Consider the \(d\)-dimensional jump diffusion
\eqref{eq:gaussian-cp-submodel} with
\(b\in\ClassB_0^J(d,r,K)\). Assume that
\(\bbE[\exp(r\abs{\eta})]<\infty\).
Then there exist positive constants \(C_\beta\) and \(C_\beta'\) such that
\[
    \beta_X(t)\le C_\beta' \exp(-C_\beta t),
    \qquad t\ge0,
\]
uniformly over \(b\in\ClassB_0^J(d,r,K)\). Here \(C_\beta\) depends only on
\(d,r,K,\iota\), while \(C_\beta'\) depends only on
\(d,r,K,\iota\) and \(\bbE[\exp(r\abs{\eta})]\).
\end{lemma}

For a fixed component \(i\in\{1,\dots,d\}\), define
\(\ClassB_i^J(r,q,\mathbf d,\mathbf t,\boldsymbol\beta,K)\) as the class
of all functions \(b\in\ClassB_0^J(d,r,K)\) satisfying the following
conditions:
\begin{enumerate}
    \item[(i)] \( \abs{b(x)-b(y)}\le K\abs{x-y} \) for all \(x,y\in\bbR^d\);
    \item[(ii)] \( f_0=b^i\one_{[0,1]^d}
        \in
        \ClassGb(q,\mathbf d,\mathbf t,\boldsymbol\beta,K) \).
\end{enumerate}
Theorem~\ref{thm:composition} and
Lemma~\ref{lem:uniform-mixing-gaussian-cp} imply the following upper bound
uniformly over this submodel, exactly as in
Oga and Koike~\cite[Proposition~3.1]{oga-koi24}.

\begin{proposition}
\label{prop:minimax-upper-gaussian-cp}
Consider the jump diffusion model \eqref{eq:gaussian-cp-submodel} with
\(b\in\ClassB_i^J(r,q,\allowbreak\mathbf d,\allowbreak\mathbf t,
\allowbreak\boldsymbol\beta,\allowbreak K)\) and
\(\bbE[\exp(r\abs{\eta})]<\infty\).
Let \(\widehat f_n\) be a minimizer of \(\ClassQ_n(f)\) over
\(f\in\ClassF(L,\mathbf p,s,F)\), where the network parameters satisfy
conditions \emph{(i)}--\emph{(iv)} in Theorem~\ref{thm:composition}. Suppose
that \(L=O(\log_2(n\Delta))\), \(n\Delta\to\infty\), and
\(n\Delta^2\to0\).
Then
\[
    \sup_{b\in\ClassB_i^J(r,q,\mathbf d,\mathbf t,\boldsymbol\beta,K)}
    \Risk_n(\widehat f_n,f_0)
    =
    O\!\left(\phi_n\log^4(n\Delta)\right)
\]
as \(n\to\infty\).
\end{proposition}

\begin{proof}
This follows from Theorem~\ref{thm:composition} and
Lemma~\ref{lem:uniform-mixing-gaussian-cp} in the same way as
Oga and Koike~\cite[Proposition~3.1]{oga-koi24}. Indeed, the Gaussian
compound-Poisson L\'evy measure has exponential moments of all orders, and
Lemma~\ref{lem:uniform-mixing-gaussian-cp} gives the constants in
Assumption~\ref{asm:mixing} uniformly over
\(\ClassB_i^J(r,q,\mathbf d,\mathbf t,\boldsymbol\beta,K)\). For an exact
minimizer, \(\Psi_n(\widehat f_n)=0\). The approximation term in
Theorem~\ref{thm:composition} is
\(O(\phi_n L\log^3(n\Delta))=O(\phi_n\log^4(n\Delta))\), because
\(L=O(\log_2(n\Delta))\). Finally, since
\[
    \phi_n
    =
    \max_{0\le j\le q}
    (n\Delta)^{-2\beta_j^*/(2\beta_j^*+t_j)}
    \ge (n\Delta)^{-1},
\]
the condition \(n\Delta^2\to0\) implies
\(\Delta=o((n\Delta)^{-1})=o(\phi_n)\). Hence the discretization term
\(C''\Delta\) in Theorem~\ref{thm:composition} is negligible.
\end{proof}

We next state the corresponding minimax lower bound. Let \(i^*\) be a rate-determining index satisfying
\[
    \frac{\beta_{i^*}^*}{2\beta_{i^*}^*+t_{i^*}}
    =
    \min_{0\le j\le q}
    \frac{\beta_j^*}{2\beta_j^*+t_j},
\]
and write
\[
    \beta^*:=\beta_{i^*},
    \qquad
    \beta^{**}:=\beta_{i^*}^*,
    \qquad
    t^*:=t_{i^*},
    \qquad
    B:=\prod_{\ell=i^*+1}^{q}(\beta_\ell\wedge1).
\]
Then \(\beta^{**}=\beta^*B\).

The following result gives a L\'evy-driven counterpart of
Oga and Koike~\cite[Theorem~3.3]{oga-koi24} on the Gaussian
compound-Poisson submodel, under the additional regularity condition
\(\beta^{**}\ge1\).

\begin{theorem}
\label{thm:minimax-lower-gaussian-cp}
Consider the jump diffusion model \eqref{eq:gaussian-cp-submodel} with
\(b\in\ClassB_i^J(r,q,\allowbreak\mathbf d,\allowbreak\mathbf t,
\allowbreak\boldsymbol\beta,\allowbreak K)\) and
\(\bbE[\exp(r\abs{\eta})]<\infty\). Suppose that
\(n\Delta\to\infty\) and \(\Delta\to0\) as \(n\to\infty\). Suppose also that
\(\beta^{**}\ge1\) and
\(t_j\le \min\{d_0,\dots,d_{j-1}\}\) for \(j=1,\dots,q\).
When \(q=0\), the second condition is understood to be void. Then there exists
\(K_0>0\) such that, for all \(K\ge K_0\),
\[
    \liminf_{n\to\infty}
    \phi_n^{-1}
    \inf_{\widehat f_n}
    \sup_{b\in\ClassB_i^J(r,q,\mathbf d,\mathbf t,\boldsymbol\beta,K)}
    \Risk_n(\widehat f_n,f_0)
    >0,
\]
where the infimum is taken over all estimators.
\end{theorem}

The condition \(\beta^{**}\ge1\) ensures that the localized packing
constructed below admits globally Lipschitz extensions with a Lipschitz
constant uniform in \(n\). The condition on
\(\mathbf d\) and \(\mathbf t\) is the same as the one imposed in the
compositional lower-bound argument of Schmidt-Hieber~\cite{sch-hie20} and used in
Oga and Koike~\cite{oga-koi24}. Proposition~\ref{prop:minimax-upper-gaussian-cp}
and Theorem~\ref{thm:minimax-lower-gaussian-cp} show that the least-squares
estimator over the neural-network class is minimax optimal over the Gaussian
compound-Poisson submodel up to the logarithmic factor \(\log^4(n\Delta)\).

\begin{remark}
\label{rem:fixed-cp-extension}
The Gaussian specification is not essential for the preceding minimax
result. The same upper and lower bounds continue to hold when
\(\nu_J\) is any fixed nonzero finite L\'evy measure satisfying
\[
    \int_{\bbR^d} z\,\nu_J(dz)=0,
    \qquad
    \int_{\bbR^d}\abs{z}^2e^{\theta_J\abs{z}}\nu_J(dz)<\infty
\]
for some \(\theta_J>0\).
Indeed, the proof uses only centering, a suitable exponential moment,
and finite jump activity for the mixing argument, while the lower-bound
construction requires only that the jump mechanism be fixed across the
drift alternatives.
\end{remark}

\section{Proof of Theorem~\ref{thm:oracle}}
\label{sec:proof-oracle}

We first define the expected empirical error by
\[
\EmpRisk_n(\widehat f_n,f_0)
:=
\bbE\Bigl[
\frac1n\sum_{k=0}^{n-1}
\bigl(\widehat f_n(X_{k\Delta})-f_0(X_{k\Delta})\bigr)^2
\Bigr].
\]
For \(\delta>0\), a subset \(\ClassG\subset\ClassF\) is called a
\(\delta\)-net of \(\ClassF\) with respect to \(\norminf{\cdot}\) if for
every \(f\in\ClassF\) there exists \(g\in\ClassG\) such that
\(\norminf{f-g}\le\delta\). The smallest possible cardinality of such a
set is denoted by \(\NN(\delta,\ClassF,\norminf{\cdot})\).

The proof starts from two general comparison bounds.

\begin{lemma}[Oga and Koike~\cite{oga-koi24}, Lemma~4.1]\label{lem:risk-to-emp-risk}
Let \(X=(X_t)_{t\ge 0}\) be a general \(d\)-dimensional stochastic process,
not necessarily satisfying \eqref{eq:model}. Let \(\delta>0\), and let
\(\NN_n\) be an integer such that
\(\NN_n\ge \NN(\delta,\ClassF,\norminf{\cdot})\vee \exp(7.4)\).
Also, let \(a_n\) be a positive number such that
\(\mu_n:=\lfloor n/(2a_n)\rfloor>0\). In addition, suppose that there is
a number \(F\ge 1\) such that \(\abs{f}\le F\) for all
\(f\in\ClassF\cup\{f_0\}\). Then, for all \(\varepsilon\in(0,1]\),
\[
\begin{aligned}
\Risk_n(\widehat f_n,f_0)
&\le
(1+\varepsilon)\EmpRisk_n(\widehat f_n,f_0)
+
\frac{24(1+\varepsilon)^2}{\varepsilon}\,
F^2\frac{\log\NN_n}{\mu_n}
\\
&\quad+4(2+\varepsilon)F^2\beta_X(a_n\Delta)
+\frac{8F^2}{n}
+4(2+\varepsilon)F\delta.
\end{aligned}
\]
\end{lemma}

\begin{lemma}\label{lem:emp-risk-bound}
Let \(X\) satisfy the L\'evy-driven SDE \eqref{eq:model}, and let \(\ClassF\) be a pointwise measurable class of real-valued functions supported in \([0,1]^d\). Assume that \(\abs{f}\le F\) for all \(f\in\ClassF\cup\{f_0\}\) with \(F\ge 1\), and write \(N_n:=\NN(\delta,\ClassF,\norminf{\cdot})\). Then, under Assumptions~\ref{asm:coefficients}, \ref{asm:sampling}, and~\ref{asm:levy-tail}, for every \(\varepsilon\in(0,1)\) there exists a constant \(C_\varepsilon>0\) depending only on
\((C_b,C_b',C_\Sigma,C_\Sigma',C_\Xi,C_\Xi',C_\nu,C_m,\lambda,\varepsilon)\)
such that
\[
\EmpRisk_n(\widehat f_n,f_0)
\le
\frac{1}{1-\varepsilon}\Psi_n^{\ClassF}(\widehat f_n)
+
\frac{1+\varepsilon}{1-\varepsilon}
\inf_{f\in\ClassF}\Risk_n(f,f_0)
+
C_\varepsilon F^2\,\gamma^{(L)}_{\delta,n},
\]
where
\[
\gamma^{(L)}_{\delta,n}
:=
\Delta+\delta
+\frac{1}{\sqrt{n\Delta}}
\int_0^{\delta}\log \NN(u,\ClassF,\norminf{\cdot})\,du
+\frac{\log N_n+1}{n\Delta}.
\]
\end{lemma}

\subsection{Proof of Lemma~\ref{lem:emp-risk-bound}}

Following Oga and Koike~\cite[Lemma~4.2]{oga-koi24}, set 
\[ 
N_n:=\NN(\delta,\ClassF,\norminf{\cdot}). 
\] 
Fix a \(\delta\)-net \(\{f_1,\dots,f_{N_n}\}\) of \(\ClassF\) with respect to \(\norminf{\cdot}\), and choose a data-dependent index \(j^*\in\{1,\dots,N_n\}\) such that 
\[ 
\norminf{\widehat f_n-f_{j^*}}\le\delta. 
\]
We now introduce the regression decomposition. Fix \(i\in\{1,\dots,d\}\).
Let \(\Sigma_i(x)\) and \(\Xi_i(x)\) denote the \(i\)-th rows of the matrices
\(\Sigma(x)\) and \(\Xi(x)\), respectively. For \(k=0,1,\dots,n-1\), define
\[
I_{k\Delta}^i
:=
\frac1\Delta
\int_{k\Delta}^{(k+1)\Delta}
\bigl(b^i(X_s)-b^i(X_{k\Delta})\bigr)\,ds,
\]
\[
\begin{aligned}
Z_{k\Delta}^{i,\prime}
&:=
\frac1\Delta
\int_{k\Delta}^{(k+1)\Delta}
\Sigma_i(X_{k\Delta})\,dW_s,\\
Z_{k\Delta}^{i,\prime\prime}
&:=
\frac1\Delta
\int_{k\Delta}^{(k+1)\Delta}
\bigl(\Sigma_i(X_{s-})-\Sigma_i(X_{k\Delta})\bigr)\,dW_s,
\end{aligned}
\]
and
\[
J_{k\Delta}^i
:=
\frac1\Delta
\int_{k\Delta}^{(k+1)\Delta}
\int_{\bbR^d\setminus\{0\}}
\Xi_i(X_{s-})z\,\widetilde N(ds,dz).
\]
Set
\[
Z_{k\Delta}^i
:=
Z_{k\Delta}^{i,\prime}+Z_{k\Delta}^{i,\prime\prime}.
\]
Then
\begin{equation}\label{eq:regression}
Y_{k\Delta}
=
b^i(X_{k\Delta})+I_{k\Delta}^i+Z_{k\Delta}^i+J_{k\Delta}^i.
\end{equation}
Compared with the diffusion setting of Oga and Koike~\cite{oga-koi24},
the present decomposition contains the additional compensated Poisson
term \(J_{k\Delta}^i\).

We next decompose the expected empirical risk into the drift, Brownian,
and jump contributions. The drift and Brownian terms are treated by the
same arguments as in Oga and Koike~\cite[Corollaries~4.1--4.2 and
Lemmas~4.8, 4.10--4.12]{oga-koi24}. The jump term is the
new part and will be handled separately.

\begin{lemma}
\label{lem:emp-risk-decomposition}
For every non-random \(\bar f\in\ClassF\),
\[
\begin{aligned}
\EmpRisk_n(\widehat f_n,f_0)
&=
\overline\Psi_n(\widehat f_n,\bar f)
+\EmpRisk_n(\bar f,f_0) \\
&\quad
+2\bbE\biggl[
\frac1n\sum_{k=0}^{n-1}
(\widehat f_n-\bar f)(X_{k\Delta})I_{k\Delta}^i
\biggr] \\
&\quad
+2\bbE\biggl[
\frac1n\sum_{k=0}^{n-1}
(\widehat f_n-f_0)(X_{k\Delta})Z_{k\Delta}^i
\biggr] \\
&\quad
+2\bbE\biggl[
\frac1n\sum_{k=0}^{n-1}
(\widehat f_n-f_0)(X_{k\Delta})J_{k\Delta}^i
\biggr],
\end{aligned}
\]
where
\[
\overline\Psi_n(\widehat f_n,\bar f)
:=
\bbE\bigl[
\ClassQ_n(\widehat f_n)-\ClassQ_n(\bar f)
\bigr].
\]
\end{lemma}

\begin{proof}
This is obtained by expanding the least-squares contrasts and using
the regression decomposition \eqref{eq:regression}, exactly as in
Oga and Koike~\cite[Lemma~4.7]{oga-koi24}. The only difference is that
the present decomposition contains the additional jump martingale term
\(J_{k\Delta}^i\), which gives the last term above.
\end{proof}

\begin{lemma}
\label{lem:diffusion-type-estimates}
The following estimates hold.

\begin{enumerate}
\item[(i)] Define \(A_t:=\{X_t\in[0,1]^d\}\) for every \(t\ge0\). Let
\[
C_{\Lip}
:=
6\bigl(C_b^2+C_\Sigma^2+C_\Xi^2C_\nu\bigr)
\exp\!\Bigl\{
6\bigl(C_b'^2+C_\Sigma'^2+C_\Xi'^2C_\nu\bigr)
\Bigr\}.
\]
For any \(k\in\{0,\dots,n-1\}\) and
\(s\in[k\Delta,(k+1)\Delta]\),
\[
\bbE\bigl[
\abs{X_s-X_{k\Delta}}^2\one_{A_{k\Delta}}
\bigr]
\le
C_{\Lip}(s-k\Delta).
\]

\item[(ii)] For every \(k\in\{0,\dots,n-1\}\),
\[
\bbE\bigl[
(I_{k\Delta}^i)^2\one_{A_{k\Delta}}
\bigr]
\le
\frac12 C_b'^2C_{\Lip}\Delta .
\]

\item[(iii)] For every \(k\in\{0,\dots,n-1\}\),
\[
\bbE\bigl[
(Z_{k\Delta}^{i,\prime\prime})^2
\one_{A_{k\Delta}}
\bigr]
\le
\frac12 C_\Sigma'^2 C_{\Lip}.
\]

\item[(iv)] For every \(\varepsilon>0\),
\[
\bbE\biggl[\biggl\lvert \frac1n\sum_{k=0}^{n-1}(\widehat f_n-\bar f)(X_{k\Delta})I_{k\Delta}^i \biggr\rvert \biggr]
\le
\frac{\varepsilon}{4}\EmpRisk_n(\widehat f_n,f_0)
+
\frac{\varepsilon}{2}\EmpRisk_n(\bar f,f_0)
+
\frac{3C_b'^2C_{\Lip}}{4\varepsilon}\Delta .
\]

\item[(v)] We have
\[
\begin{aligned}
&\bbE\biggl[
\biggl\lvert
\frac1n\sum_{k=0}^{n-1}
(f_{j^*}-f_0)(X_{k\Delta})Z_{k\Delta}^i
\biggr\rvert
\biggr]
\le
\biggl\{
\frac{4C_\Sigma^2}{n\Delta}
\bigl(2\log N_n+\log 2\bigr)
\biggr\}^{1/2} \\
&\hspace{35mm}\times
\biggl\{
\EmpRisk_n(\widehat f_n,f_0)
+4F\delta
+
\frac{2F^2C_\Sigma'^2C_{\Lip}}{C_\Sigma^2}\Delta
\biggr\}^{1/2}.
\end{aligned}
\]

\item[(vi)] We have
\[
\bbE\biggl[\biggl\lvert \frac1n\sum_{k=0}^{n-1}(\widehat f_n-f_{j^*})(X_{k\Delta})Z_{k\Delta}^{i,\prime\prime} \biggr\rvert \biggr]
\le
\delta
\biggl(
\frac12 C_\Sigma'^2C_{\Lip}
\biggr)^{1/2}.
\]

\item[(vii)] There exists a constant \(C_Z>0\), depending only on
\(C_\Sigma\), such that
\[
\bbE\biggl[\biggl\lvert \frac1n\sum_{k=0}^{n-1}(\widehat f_n-f_{j^*})(X_{k\Delta})Z_{k\Delta}^{i,\prime} \biggr\rvert \biggr]
\le
\frac{C_Z}{\sqrt{n\Delta}}
\int_0^\delta
\sqrt{
\log \NN(u,\ClassF,\norminf{\cdot})
}\,du .
\]
\end{enumerate}
\end{lemma}

\begin{proof}
We first prove (i). Using the It\^o isometry and the \(L^2\)-isometry for stochastic
integrals with respect to compensated Poisson random measures
(Applebaum~\cite[Lemma~4.2.2, its subsequent extension, and
Section~4.3.2]{app09}),
\[
\begin{aligned}
&\bbE\bigl[
\abs{X_s-X_{k\Delta}}^2\one_{A_{k\Delta}}
\bigr] \\
&=
\bbE\Biggl[
\sum_{i=1}^d
\Biggl(
\int_{k\Delta}^{s} b^i(X_{u-})\,du
+
\int_{k\Delta}^{s} \Sigma^i(X_{u-})\,dW_u
\\[-0.2em]
&\hspace{8em}
+
\int_{k\Delta}^{s}\int_{\bbR^d\setminus\{0\}}
\Xi^i(X_{u-})z\,\widetilde N(du,dz)
\Biggr)^2
\one_{A_{k\Delta}}
\Biggr] \\
&\le
3\bbE\Biggl[
\sum_{i=1}^d
\Biggl(
\int_{k\Delta}^{s} b^i(X_{u-})\one_{A_{k\Delta}}\,du
\Biggr)^2
+
\sum_{i=1}^d
\int_{k\Delta}^{s}
\Sigma^i(X_{u-})\Sigma^i(X_{u-})^\top
\one_{A_{k\Delta}}\,du
\\[-0.2em]
&\hspace{4em}
+
\sum_{i=1}^d
\int_{k\Delta}^{s}\int_{\bbR^d\setminus\{0\}}
\bigl(\Xi^i(X_{u-})z\bigr)^2
\one_{A_{k\Delta}}\,\nu(dz)\,du
\Biggr] \\
&\le
3\bbE\Biggl[
\int_{k\Delta}^{s}
\sum_{i=1}^d
\Bigl\{
\bigl(b^i(X_{u-})\bigr)^2
+
\Sigma^i(X_{u-})\Sigma^i(X_{u-})^\top
\\[-0.2em]
&\hspace{11em}
+
C_\nu
\Xi^i(X_{u-})\Xi^i(X_{u-})^\top
\Bigr\}
\one_{A_{k\Delta}}\,du
\Biggr],
\end{aligned}
\]
where the last inequality follows from the Schwarz inequality and
\(s-k\Delta\le1\), and from
\[
\int_{\bbR^d\setminus\{0\}}
\bigl(\Xi^i(X_{u-})z\bigr)^2\,\nu(dz)
\le
C_\nu\,
\Xi^i(X_{u-})\Xi^i(X_{u-})^\top .
\]
Additionally, for any \(u\in[k\Delta,(k+1)\Delta]\),
\[
\begin{aligned}
\bigl(b^i(X_{u-})\bigr)^2
&=
\bigl(b^i(X_{k\Delta})
+
(b^i(X_{u-})-b^i(X_{k\Delta}))\bigr)^2
\\
&\le
2\bigl(b^i(X_{k\Delta})\bigr)^2
+
2\bigl(b^i(X_{u-})-b^i(X_{k\Delta})\bigr)^2,
\end{aligned}
\]
and
\[
\begin{aligned}
\Sigma^i(X_{u-})\Sigma^i(X_{u-})^\top
&=
\sum_j
\bigl(
\Sigma^{ij}(X_{k\Delta})
+
(\Sigma^{ij}(X_{u-})-\Sigma^{ij}(X_{k\Delta}))
\bigr)^2 \\
&\le
2\sum_j\bigl(\Sigma^{ij}(X_{k\Delta})\bigr)^2
+
2\sum_j
\bigl(\Sigma^{ij}(X_{u-})-\Sigma^{ij}(X_{k\Delta})\bigr)^2,
\end{aligned}
\]
and similarly,
\[
\begin{aligned}
\Xi^i(X_{u-})\Xi^i(X_{u-})^\top
&\le
2\sum_j\bigl(\Xi^{ij}(X_{k\Delta})\bigr)^2
+
2\sum_j
\bigl(\Xi^{ij}(X_{u-})-\Xi^{ij}(X_{k\Delta})\bigr)^2 .
\end{aligned}
\]
Using these inequalities and Assumption~\ref{asm:coefficients}, we conclude
\[
\begin{aligned}
&\bbE\bigl[
\abs{X_s-X_{k\Delta}}^2\one_{A_{k\Delta}}
\bigr] \\
&\le
6\bbE\biggl[
\int_{k\Delta}^{s}
\Bigl\{
C_b^2+C_\Sigma^2+C_\Xi^2C_\nu
+
\bigl(C_b'^2+C_\Sigma'^2+C_\Xi'^2C_\nu\bigr)
\abs{X_{u-}-X_{k\Delta}}^2
\one_{A_{k\Delta}}
\Bigr\}
\,du
\biggr] \\
&=
6\bigl(C_b^2+C_\Sigma^2+C_\Xi^2C_\nu\bigr)(s-k\Delta) \\
&\qquad
+
6\bigl(C_b'^2+C_\Sigma'^2+C_\Xi'^2C_\nu\bigr)
\int_{k\Delta}^{s}
\bbE\bigl[
\abs{X_u-X_{k\Delta}}^2
\one_{A_{k\Delta}}
\bigr]\,du .
\end{aligned}
\]
In the last line, we used the fact that \(X_{u-}=X_u\) for Lebesgue-a.e.
\(u\), since \(X\) is c\`adl\`ag. Therefore, by Gronwall's inequality,
\[
\begin{aligned}
&
\bbE\bigl[
\abs{X_s-X_{k\Delta}}^2\one_{A_{k\Delta}}
\bigr]
\\
&\qquad\le
6\bigl(C_b^2+C_\Sigma^2+C_\Xi^2C_\nu\bigr)
(s-k\Delta)
\exp\!\Bigl\{
6\bigl(C_b'^2+C_\Sigma'^2+C_\Xi'^2C_\nu\bigr)(s-k\Delta)
\Bigr\}.
\end{aligned}
\]
Since \(s-k\Delta\le\Delta\le1\) by Assumption~\ref{asm:sampling},
(i) follows.

We next prove (ii)--(vii). The proof is the same as the diffusion proof
of Oga and Koike~\cite[Corollaries~4.1--4.2 and Lemmas~4.8,
4.10--4.12]{oga-koi24}, with the localized increment estimate in their
argument replaced by (i). Thus the constant appearing in those estimates
is replaced by the present \(C_{\Lip}\).

More precisely, (ii) and (iii) are the counterparts of their
Corollaries~4.1 and 4.2. Then (iv) follows from (ii) exactly as in their
Lemma~4.8. Finally, the Brownian estimates (v)--(vii) are obtained as in
their Lemmas~4.10--4.12; the only resulting change is that the Brownian
remainder estimate uses (iii), hence the displayed \(C_{\Lip}\)-dependent
constants.
\end{proof}

We now estimate the jump contribution in
Lemma~\ref{lem:emp-risk-decomposition}. The drift and Brownian terms in that
decomposition are controlled by Lemma~\ref{lem:diffusion-type-estimates}, so it
remains to handle the term containing \(J_{k\Delta}^i\).

We split the jump stochastic term as
\begin{equation}\label{eq:jump-net-osc-decomposition}
\begin{aligned}
\frac1n\sum_{k=0}^{n-1}
(\widehat f_n-f_0)(X_{k\Delta})J_{k\Delta}^i
&=
\frac1n\sum_{k=0}^{n-1}
(f_{j^*}-f_0)(X_{k\Delta})J_{k\Delta}^i  \\
&\quad+
\frac1n\sum_{k=0}^{n-1}
(\widehat f_n-f_{j^*})(X_{k\Delta})J_{k\Delta}^i .
\end{aligned}
\end{equation}
The first term is a finite-net term, while the second is a local oscillation
term. We bound the former by an exponential-supermartingale argument combined
with a union bound over the net, and the latter by deriving a \(\psi_1\)
increment bound and then applying a chaining argument.

The sub-exponential moment assumption on the L\'evy measure is used precisely in these
exponential estimates: it guarantees the finiteness of the exponential
compensators. Recall that, for \(a\in[0,\lambda)\),
\[
\Lambda(a):=
\int_{\bbR^d\setminus\{0\}}\abs{z}^2e^{a\abs{z}}\,\nu(dz)<\infty
\]
by Assumption~\ref{asm:levy-tail}; see also
Remark~\ref{rem:Lambda-finite}. This quantity enters the compensator bounds
through \(e^x-1-x\le x^2e^{\abs{x}}\) for \(x\in\bbR\).

\begin{lemma}
\label{lem:jump-exponential-local-martingale}
Let \(h=h(s,z)\) be a predictable integrand on
\([0,\infty)\times(\bbR^d\setminus\{0\})\), and fix \(\varepsilon>0\).
For \(t\ge0\), define
\[
M_t(h):=
\int_0^t\int_{\bbR^d\setminus\{0\}}
h(s,z)\,\widetilde N(ds,dz),
\]
\[
A_{\varepsilon,t}(h):=
\int_0^t\int_{\bbR^d\setminus\{0\}}
\bigl(e^{\varepsilon h(s,z)}-1-\varepsilon h(s,z)\bigr)
\,\nu(dz)\,ds,
\]
and
\[
Z_t^\varepsilon(h):=
\exp\bigl(\varepsilon M_t(h)-A_{\varepsilon,t}(h)\bigr).
\]
Assume that, for every \(t\ge0\),
\[
I_{\varepsilon,t}(h):=
\int_0^t\int_{\bbR^d\setminus\{0\}}
\abs{h(s,z)}^2 e^{\varepsilon\abs{h(s,z)}}\,\nu(dz)\,ds
<\infty
\]
almost surely. Then \(M(h)\) is a locally square-integrable local
martingale, \(A_{\varepsilon,t}(h)<\infty\) for every \(t\ge0\) almost surely,
and both \(Z^\varepsilon(h)\) and \(Z^\varepsilon(-h)\) are nonnegative local
martingales. In particular, they are supermartingales and
\[
\bbE[Z_t^\varepsilon(h)]\le1,
\qquad
\bbE[Z_t^\varepsilon(-h)]\le1,
\qquad t\ge0.
\]
Moreover,
\[
0\le A_{\varepsilon,t}(h)
\le
\varepsilon^2 I_{\varepsilon,t}(h),
\]
and
\[
A_{\varepsilon,t}(h)+A_{\varepsilon,t}(-h)
\le
\varepsilon^2 I_{\varepsilon,t}(h).
\]
\end{lemma}

\begin{proof}
Since \(e^{\varepsilon\abs{h}}\ge1\), the assumption implies that, for
every \(t\ge0\),
\[
\int_0^t\int_{\bbR^d\setminus\{0\}}
\abs{h(s,z)}^2\,\nu(dz)\,ds<\infty
\]
almost surely. Hence \(M(h)\) is a well-defined compensated Poisson
integral and is a locally square-integrable local martingale.

For \(u\in\bbR\),
\[
0\le e^u-1-u\le u^2e^{\abs{u}},
\qquad
0\le e^u+e^{-u}-2\le u^2e^{\abs{u}}.
\]
Therefore,
\[
0\le A_{\varepsilon,t}(h)
\le
\varepsilon^2 I_{\varepsilon,t}(h)<\infty,
\]
and
\[
A_{\varepsilon,t}(h)+A_{\varepsilon,t}(-h)
=
\int_0^t\int_{\bbR^d\setminus\{0\}}
\bigl(e^{\varepsilon h(s,z)}
+e^{-\varepsilon h(s,z)}-2\bigr)\,\nu(dz)\,ds
\le
\varepsilon^2 I_{\varepsilon,t}(h).
\]

It remains to prove the exponential local martingale property. We connect
the present all-space compensated notation with Applebaum's ordinary
exponential criterion
\cite[Section~5.2, Corollary~5.2.2]{app09}. Put
\[
Y_t:=\varepsilon M_t(h)-A_{\varepsilon,t}(h).
\]
Splitting the jump region into small and large jumps gives
\begin{align*}
Y_t
&=
\int_0^t\int_{\{\abs{z}<1\}}
\varepsilon h(s,z)\,\widetilde N(ds,dz)
+
\int_0^t\int_{\{\abs{z}\ge1\}}
\varepsilon h(s,z)\,N(ds,dz) \\
&\quad
-\int_0^t\int_{\{\abs{z}<1\}}
\bigl(e^{\varepsilon h(s,z)}-1-\varepsilon h(s,z)\bigr)
\,\nu(dz)\,ds \\
&\quad
-\int_0^t\int_{\{\abs{z}\ge1\}}
\bigl(e^{\varepsilon h(s,z)}-1\bigr)
\,\nu(dz)\,ds .
\end{align*}
Here we used
\[
\int_{\{\abs{z}\ge1\}}\varepsilon h(s,z)\,\widetilde N(ds,dz)
=
\int_{\{\abs{z}\ge1\}}\varepsilon h(s,z)\,N(ds,dz)
-
\int_{\{\abs{z}\ge1\}}\varepsilon h(s,z)\,\nu(dz)\,ds.
\]

Thus \(Y\) is in Applebaum's small-jump/large-jump form with no Brownian
part,
\[
H(s,z)=\varepsilon h(s,z)
\quad\text{on }\{\abs{z}<1\},
\qquad
K(s,z)=\varepsilon h(s,z)
\quad\text{on }\{\abs{z}\ge1\},
\]
and finite-variation coefficient
\[
G(s)=
-\int_{\{\abs{z}<1\}}
\bigl(e^{\varepsilon h(s,z)}-1-\varepsilon h(s,z)\bigr)\nu(dz)
-\int_{\{\abs{z}\ge1\}}
\bigl(e^{\varepsilon h(s,z)}-1\bigr)\nu(dz).
\]
The small-jump \(L^2\) condition follows from the square-integrability of
\(h\). The large-jump integrability follows from
\(\nu(\{\abs{z}\ge1\})<\infty\), the square-integrability of \(h\), and
\(I_{\varepsilon,t}(h)<\infty\); indeed, splitting according to
\(\{\varepsilon\abs{h}\le1\}\) and \(\{\varepsilon\abs{h}>1\}\) gives
\[
\int_0^t\int_{\{\abs{z}\ge1\}}
\abs{e^{\varepsilon h(s,z)}-1}\,\nu(dz)\,ds<\infty.
\]

Moreover, the drift cancellation condition in
\cite[Corollary~5.2.2]{app09} is satisfied identically:
\begin{align*}
& G(s)
+\int_{\{\abs{z}<1\}}
\bigl(e^{H(s,z)}-1-H(s,z)\bigr)\nu(dz)
+\int_{\{\abs{z}\ge1\}}
\bigl(e^{K(s,z)}-1\bigr)\nu(dz) \\
&\qquad =0
\end{align*}
for Lebesgue-a.e. \(s\), almost surely. Hence
\(e^{Y_t}=Z_t^\varepsilon(h)\) is a local martingale. Applying the same
argument to \(-h\) gives that \(Z^\varepsilon(-h)\) is also a local
martingale. Since both processes are nonnegative, they are
supermartingales, and therefore
\[
\bbE[Z_t^\varepsilon(h)]\le Z_0^\varepsilon(h)=1,
\qquad
\bbE[Z_t^\varepsilon(-h)]\le Z_0^\varepsilon(-h)=1.
\]
\end{proof}

We first control the finite-net term, namely the first term on the
right-hand side of \eqref{eq:jump-net-osc-decomposition}.

\begin{lemma}\label{lem:jump-net-main-term}
For every \(\varepsilon\in(0,1)\), set
\[
\varepsilon_J:=
\min\left\{
\frac{\lambda\wedge1}{4FC_\Xi},
\frac{\varepsilon}{8C_\Xi^2\Lambda\!\left(\frac{\lambda\wedge1}{2}\right)}
\right\}.
\]
Then
\[
\bbE\Bigl[
\frac1n\sum_{k=0}^{n-1}(f_{j^*}-f_0)(X_{k\Delta})J_{k\Delta}^i
\Bigr]
\le
\frac{\log N_n}{\varepsilon_J n\Delta}
+
\frac{\varepsilon}{8}
\Bigl(
\EmpRisk_n(\widehat f_n,f_0)+4F\delta
\Bigr).
\]
\end{lemma}

\begin{proof}
Fix \(f\in\ClassF\) and define
\[
h_f(s,z):=
\sum_{k=0}^{n-1}
(f-f_0)(X_{k\Delta})\Xi_i(X_{s-})z\,
\one_{(k\Delta,(k+1)\Delta]}(s),
\qquad
0\le s\le n\Delta .
\]
Since \(\norminf{f-f_0}\le 2F\), we have
\[
\abs{\varepsilon_J h_f(s,z)}
\le
2\varepsilon_J FC_\Xi\abs{z}
\le
\frac{\lambda\wedge1}{2}\abs{z}.
\]
Consequently,
\begin{align*}
I_{\varepsilon_J,n\Delta}(h_f)
&=
\int_0^{n\Delta}\int_{\bbR^d\setminus\{0\}}
\abs{h_f(s,z)}^2e^{\varepsilon_J\abs{h_f(s,z)}}\,\nu(dz)\,ds \\
&\le
C_\Xi^2
\Lambda\!\left(\frac{\lambda\wedge1}{2}\right)
\Delta
\sum_{k=0}^{n-1}(f-f_0)^2(X_{k\Delta})
<\infty
\end{align*}
almost surely. Hence Lemma~\ref{lem:jump-exponential-local-martingale}
applies to \(h_f\).

In particular, for each net element \(f_j\),
\[
\bbE\left[
\exp\bigl(
\varepsilon_J M_{n\Delta}(h_{f_j})
-
A_{\varepsilon_J,n\Delta}(h_{f_j})
\bigr)
\right]
\le 1.
\]
Since \(j^*\in\{1,\dots,N_n\}\),
\[
\exp\bigl(
\varepsilon_J M_{n\Delta}(h_{f_{j^*}})
-
A_{\varepsilon_J,n\Delta}(h_{f_{j^*}})
\bigr)
\le
\sum_{j=1}^{N_n}
\exp\bigl(
\varepsilon_J M_{n\Delta}(h_{f_j})
-
A_{\varepsilon_J,n\Delta}(h_{f_j})
\bigr).
\]
Therefore
\[
\bbE\left[
\exp\bigl(
\varepsilon_J M_{n\Delta}(h_{f_{j^*}})
-
A_{\varepsilon_J,n\Delta}(h_{f_{j^*}})
\bigr)
\right]
\le N_n.
\]
By Jensen's inequality,
\[
\varepsilon_J\bbE[M_{n\Delta}(h_{f_{j^*}})]
-
\bbE[A_{\varepsilon_J,n\Delta}(h_{f_{j^*}})]
\le
\log N_n.
\]
Moreover, Lemma~\ref{lem:jump-exponential-local-martingale} and the
above bound on \(I_{\varepsilon_J,n\Delta}(h_f)\) give
\[
A_{\varepsilon_J,n\Delta}(h_f)
\le
\varepsilon_J^2 C_\Xi^2
\Lambda\!\left(\frac{\lambda\wedge1}{2}\right)
\Delta
\sum_{k=0}^{n-1}(f-f_0)^2(X_{k\Delta}).
\]
Applying this with \(f=f_{j^*}\) yields
\[
\bbE[A_{\varepsilon_J,n\Delta}(h_{f_{j^*}})]
\le
\varepsilon_J^2 C_\Xi^2
\Lambda\!\left(\frac{\lambda\wedge1}{2}\right)
n\Delta
\Bigl(
\EmpRisk_n(\widehat f_n,f_0)+4F\delta
\Bigr),
\]
because
\((f_{j^*}-f_0)^2 \le (\widehat f_n-f_0)^2+4F\delta\).

Finally, by the definition of \(h_f\) and \(J_{k\Delta}^i\),
\[
M_{n\Delta}(h_f)
=
\Delta
\sum_{k=0}^{n-1}
(f-f_0)(X_{k\Delta})J_{k\Delta}^i.
\]
Thus
\begin{align*}
&\bbE\Bigl[
\frac1n\sum_{k=0}^{n-1}
(f_{j^*}-f_0)(X_{k\Delta})J_{k\Delta}^i
\Bigr] \\
&\quad =
\frac{1}{n\Delta}\bbE[M_{n\Delta}(h_{f_{j^*}})] \\
&\quad \le
\frac{\log N_n}{\varepsilon_J n\Delta}
+
\varepsilon_J C_\Xi^2
\Lambda\!\left(\frac{\lambda\wedge1}{2}\right)
\Bigl(
\EmpRisk_n(\widehat f_n,f_0)+4F\delta
\Bigr).
\end{align*}
Since
\(\varepsilon_J C_\Xi^2 \Lambda\!\left((\lambda\wedge1)/2\right) \le \varepsilon/8\),
the desired bound follows.
\end{proof}

We next control the second term on the right-hand side of
\eqref{eq:jump-net-osc-decomposition}. This is the local oscillation term around
the net representative \(f_{j^*}\). To prepare for the subsequent chaining
argument, we first prove a fixed-pair tail estimate for deterministic
\(f,g\in\ClassF\), which yields the required \(\psi_1\)-increment bound.

\begin{lemma}\label{lem:jump-increment-tail}
For non-random
\(f,g\in\ClassF\) and
\((s,z)\in[0,n\Delta]\times(\bbR^d\setminus\{0\})\), define
\[
h_{f,g}(s,z):=
\frac1{\sqrt{n\Delta}}
\sum_{k=0}^{n-1}
(f-g)(X_{k\Delta})\Xi_i(X_{s-})z\,
\one_{(k\Delta,(k+1)\Delta]}(s).
\]
If \(\norminf{f-g}=0\), then \(h_{f,g}\equiv0\). Otherwise, set
\[
C_{\lambda,\Xi}:=\frac{\lambda\wedge1}{2C_\Xi},
\qquad
B_{\lambda,\Xi}:=
C_{\lambda,\Xi}^2C_\Xi^2
\Lambda\!\left(\frac{\lambda\wedge1}{2}\right),
\qquad
\varepsilon_{f,g}:=
\frac{C_{\lambda,\Xi}}{\norminf{f-g}}.
\]
Then
\[
0\le
\varepsilon_{f,g}^{-1}
A_{\varepsilon_{f,g},n\Delta}(h_{f,g})
\le
C_{\lambda,\Xi}C_\Xi^2
\Lambda\!\left(\frac{\lambda\wedge1}{2}\right)
\norminf{f-g},
\]
and, for every \(x>0\),
\[
\bbP\left(
M_{n\Delta}(h_{f,g})
-
\varepsilon_{f,g}^{-1}
A_{\varepsilon_{f,g},n\Delta}(h_{f,g})
\ge x
\right)
\le
\exp\!\left(
-\frac{C_{\lambda,\Xi}}{\norminf{f-g}}x
\right),
\]
\[
\bbP\left(
M_{n\Delta}(h_{f,g})
-
\varepsilon_{f,g}^{-1}
A_{\varepsilon_{f,g},n\Delta}(h_{f,g})
\le -x
\right)
\le
\exp(B_{\lambda,\Xi})
\exp\!\left(
-\frac{C_{\lambda,\Xi}}{\norminf{f-g}}x
\right).
\]
Consequently,
\[
\bbP\left(
\abs{
M_{n\Delta}(h_{f,g})
-
\varepsilon_{f,g}^{-1}
A_{\varepsilon_{f,g},n\Delta}(h_{f,g})
}
\ge x
\right)
\le
\bigl(1+\exp(B_{\lambda,\Xi})\bigr)
\exp\!\left(
-\frac{C_{\lambda,\Xi}}{\norminf{f-g}}x
\right).
\]
\end{lemma}

\begin{proof}
If \(\norminf{f-g}=0\), then \(h_{f,g}\equiv0\), and the claim is
trivial. Hence assume \(\norminf{f-g}>0\).

Since \(f,g\) are non-random, \(h_{f,g}\) is predictable. Moreover,
\[
\abs{h_{f,g}(s,z)}
\le
\frac{\norminf{f-g}C_\Xi}{\sqrt{n\Delta}}\abs{z},
\]
and hence, by Assumption~\ref{asm:sampling},
\[
\abs{\varepsilon_{f,g}h_{f,g}(s,z)}
\le
\frac{\lambda\wedge1}{2\sqrt{n\Delta}}\abs{z}
\le
\frac{\lambda\wedge1}{2}\abs{z}.
\]
Thus
\begin{align*}
I_{\varepsilon_{f,g},n\Delta}(h_{f,g})
&=
\int_0^{n\Delta}\int_{\bbR^d\setminus\{0\}}
\abs{h_{f,g}(s,z)}^2
e^{\varepsilon_{f,g}\abs{h_{f,g}(s,z)}}\,\nu(dz)\,ds \\
&\le
\frac{\norminf{f-g}^2C_\Xi^2}{n\Delta}
\int_0^{n\Delta}\int_{\bbR^d\setminus\{0\}}
\abs{z}^2
e^{\frac{\lambda\wedge1}{2}\abs{z}}\,\nu(dz)\,ds \\
&=
\norminf{f-g}^2C_\Xi^2
\Lambda\!\left(\frac{\lambda\wedge1}{2}\right)
<\infty .
\end{align*}
Hence Lemma~\ref{lem:jump-exponential-local-martingale} applies to
\(h_{f,g}\) with \(\varepsilon=\varepsilon_{f,g}\).

By the compensator bound in
Lemma~\ref{lem:jump-exponential-local-martingale},
\[
A_{\varepsilon_{f,g},n\Delta}(h_{f,g})
\le
\varepsilon_{f,g}^2
I_{\varepsilon_{f,g},n\Delta}(h_{f,g}).
\]
Therefore
\[
\varepsilon_{f,g}^{-1}
A_{\varepsilon_{f,g},n\Delta}(h_{f,g})
\le
\varepsilon_{f,g}
\norminf{f-g}^2C_\Xi^2
\Lambda\!\left(\frac{\lambda\wedge1}{2}\right)
=
C_{\lambda,\Xi}C_\Xi^2
\Lambda\!\left(\frac{\lambda\wedge1}{2}\right)
\norminf{f-g}.
\]

For the upper tail, the same lemma gives
\[
\bbE\left[
\exp\left(
\varepsilon_{f,g}M_{n\Delta}(h_{f,g})
-
A_{\varepsilon_{f,g},n\Delta}(h_{f,g})
\right)
\right]\le1.
\]
Hence Markov's inequality yields
\[
\bbP\left(
M_{n\Delta}(h_{f,g})
-
\varepsilon_{f,g}^{-1}
A_{\varepsilon_{f,g},n\Delta}(h_{f,g})
\ge x
\right)
\le
e^{-\varepsilon_{f,g}x}.
\]

For the lower tail, applying
Lemma~\ref{lem:jump-exponential-local-martingale} to \(-h_{f,g}\) gives
\[
\bbE\left[
\exp\left(
-\varepsilon_{f,g}M_{n\Delta}(h_{f,g})
-
A_{\varepsilon_{f,g},n\Delta}(-h_{f,g})
\right)
\right]\le1.
\]
Moreover, by the same lemma,
\[
A_{\varepsilon_{f,g},n\Delta}(h_{f,g})
+
A_{\varepsilon_{f,g},n\Delta}(-h_{f,g})
\le
\varepsilon_{f,g}^2
I_{\varepsilon_{f,g},n\Delta}(h_{f,g})
\le
B_{\lambda,\Xi}.
\]
Consequently,
\begin{align*}
&\bbE\left[
\exp\left(
-\varepsilon_{f,g}
\left\{
M_{n\Delta}(h_{f,g})
-
\varepsilon_{f,g}^{-1}
A_{\varepsilon_{f,g},n\Delta}(h_{f,g})
\right\}
\right)
\right] \\
&\quad
=
\bbE\left[
\exp\left(
-\varepsilon_{f,g}M_{n\Delta}(h_{f,g})
+
A_{\varepsilon_{f,g},n\Delta}(h_{f,g})
\right)
\right] \\
&\quad
\le
\exp(B_{\lambda,\Xi}).
\end{align*}
Another application of Markov's inequality gives
\[
\bbP\left(
M_{n\Delta}(h_{f,g})
-
\varepsilon_{f,g}^{-1}
A_{\varepsilon_{f,g},n\Delta}(h_{f,g})
\le -x
\right)
\le
\exp(B_{\lambda,\Xi})e^{-\varepsilon_{f,g}x}.
\]
Substituting
\(\varepsilon_{f,g} = C_{\lambda,\Xi}/\norminf{f-g}\)
gives the one-sided bounds, and the two-sided estimate follows from the
union bound.
\end{proof}

To pass from the fixed-pair increment bound above to a uniform local
estimate, we use a \(\psi_1\)-chaining inequality stated and proved in
\ref{app:psi1-chaining}; see
Lemma~\ref{lem:psi1-chaining}.
For \(x\ge0\), let \(\psi_1(x):=e^x-1\), and, for a real-valued random
variable \(Y\), define the associated Orlicz norm by
\[
    \norm{Y}_{\psi_1}
    :=
    \inf\left\{
        c>0:
        \bbE\left[
            \psi_1\left(\frac{\abs{Y}}{c}\right)
        \right]
        \le 1
    \right\},
\]
with the convention that the infimum over the empty set is \(+\infty\).

\begin{lemma}\label{lem:jump-chaining}
There exists a constant \(C_J>0\),
depending only on \(C_\Xi\), \(C_m\), and \(\lambda\), such that
\[
\bbE\Bigl[
\Bigl\lvert
\frac1n\sum_{k=0}^{n-1}(\widehat f_n-f_{j^*})(X_{k\Delta})J_{k\Delta}^i
\Bigr\rvert
\Bigr]
\le
\frac{C_J}{\sqrt{n\Delta}}
\int_0^\delta \log \NN(u,\ClassF,\norminf{\cdot})\,du.
\]
\end{lemma}

\begin{proof}
Fix non-random \(f,g\in\ClassF\). If \(\norminf{f-g}=0\), then
\(h_{f,g}\equiv0\), and hence
\(M_{n\Delta}(h_{f,g})=0\).
Thus the desired \(\psi_1\)-increment bound is trivial in this case.
Hence assume \(\norminf{f-g}>0\), and set
\(\varepsilon_{f,g}:= C_{\lambda,\Xi}/\norminf{f-g}\).
Define
\[
Y(f,g):=
M_{n\Delta}(h_{f,g})
-
\varepsilon_{f,g}^{-1}
A_{\varepsilon_{f,g},n\Delta}(h_{f,g}).
\]
Lemma~\ref{lem:jump-increment-tail} implies that, for every \(x>0\),
\[
\bbP\bigl(\abs{Y(f,g)}\ge x\bigr)
\le
\bigl(1+\exp(B_{\lambda,\Xi})\bigr)
\exp\!\left(
-\frac{C_{\lambda,\Xi}}{\norminf{f-g}}x
\right).
\]
Applying van der Vaart and Wellner
\cite[Lemma~2.2.1]{van-wel96} with
\(C=C_{\lambda,\Xi}/\norminf{f-g}\) and
\(K=1+\exp(B_{\lambda,\Xi})\), we obtain
\begin{equation}\label{eq:jump-Y-psi1}
\norm{Y(f,g)}_{\psi_1}
\le
\frac{2+\exp(B_{\lambda,\Xi})}{C_{\lambda,\Xi}}
\norminf{f-g}.
\end{equation}
Next, by Lemma~\ref{lem:jump-increment-tail},
\[
0\le
\varepsilon_{f,g}^{-1}
A_{\varepsilon_{f,g},n\Delta}(h_{f,g})
\le
C_{\lambda,\Xi}C_\Xi^2
\Lambda\!\left(\frac{\lambda\wedge1}{2}\right)
\norminf{f-g}.
\]
Since \(\psi_1(\log2)=1\), the definition of the Orlicz norm gives
\(\norm{U}_{\psi_1}\le B/\log2\) whenever \(\abs{U}\le B\) a.s.
Therefore,
\begin{equation}\label{eq:jump-A-psi1}
\norm{
\varepsilon_{f,g}^{-1}
A_{\varepsilon_{f,g},n\Delta}(h_{f,g})
}_{\psi_1}
\le
\frac{C_{\lambda,\Xi}C_\Xi^2}{\log2}
\Lambda\!\left(\frac{\lambda\wedge1}{2}\right)
\norminf{f-g}.
\end{equation}
Since
\[
M_{n\Delta}(h_{f,g})
=
Y(f,g)
+
\varepsilon_{f,g}^{-1}
A_{\varepsilon_{f,g},n\Delta}(h_{f,g}),
\]
the triangle inequality for the Orlicz norm, together with
\eqref{eq:jump-Y-psi1} and \eqref{eq:jump-A-psi1}, yields
\[
\norm{M_{n\Delta}(h_{f,g})}_{\psi_1}
\le
\left(
\frac{2+\exp(B_{\lambda,\Xi})}{C_{\lambda,\Xi}}
+
\frac{C_{\lambda,\Xi}C_\Xi^2}{\log2}
\Lambda\!\left(\frac{\lambda\wedge1}{2}\right)
\right)
\norminf{f-g}.
\]

Now define the process indexed by \(f\in\ClassF\),
\[
\widetilde Z^{J}(f)
:=
\frac1{\sqrt{n\Delta}}
\sum_{k=0}^{n-1}f(X_{k\Delta})
\int_{k\Delta}^{(k+1)\Delta}
\int_{\bbR^d\setminus\{0\}}
\Xi_i(X_{u-})z\,\widetilde N(du,dz).
\]
If
\[
    \int_0^\delta
    \log \NN(u,\ClassF,\norminf{\cdot})\,du
    =\infty,
\]
the desired bound is trivial. Otherwise
\((\ClassF,\norminf{\cdot})\) is separable, and the sample paths of
\(\widetilde Z^J\) are \(\norminf{\cdot}\)-continuous. Hence
\((\widetilde Z^J(f))_{f\in\ClassF}\) is separable.
Then
\[
\widetilde Z^{J}(f)-\widetilde Z^{J}(g)
=
M_{n\Delta}(h_{f,g}),
\]
and hence the previous estimate shows that the process
\((\widetilde Z^{J}(f))_{f\in\ClassF}\) has \(\psi_1\)-increments
controlled by \(\norminf{f-g}\). Applying Lemma~\ref{lem:psi1-chaining}
with the semimetric \(d(f,g):=\norminf{f-g}\) yields
\[
\bbE\Bigl[
\sup_{\norminf{f-g}\le\delta}
\abs{\widetilde Z^{J}(f)-\widetilde Z^{J}(g)}
\Bigr]
\le
C_J
\int_0^\delta \log \NN(u,\ClassF,\norminf{\cdot})\,du,
\]
for a constant \(C_J>0\) depending only on \(C_\Xi\), \(C_m\), and
\(\lambda\).

Finally, using the property of the chosen index \(j^*\) that
\(\norminf{\widehat f_n-f_{j^*}}\le\delta\), we have
\begin{align*}
&\bbE\Bigl[
\Bigl\lvert
\frac1n\sum_{k=0}^{n-1}
(\widehat f_n-f_{j^*})(X_{k\Delta})J_{k\Delta}^i
\Bigr\rvert
\Bigr] \\
&\quad =
\frac1{\sqrt{n\Delta}}
\bbE\Bigl[
\abs{
\widetilde Z^{J}(\widehat f_n)-\widetilde Z^{J}(f_{j^*})
}
\Bigr] \\
&\quad \le
\frac1{\sqrt{n\Delta}}
\bbE\Bigl[
\sup_{\norminf{f-g}\le\delta}
\abs{\widetilde Z^{J}(f)-\widetilde Z^{J}(g)}
\Bigr] \\
&\quad \le
\frac{C_J}{\sqrt{n\Delta}}
\int_0^\delta \log \NN(u,\ClassF,\norminf{\cdot})\,du.
\end{align*}
This proves the claim.
\end{proof}

\begin{proof}[Proof of Lemma~\ref{lem:emp-risk-bound}]
Fix a non-random \(\bar f\in\ClassF\). By Lemma~\ref{lem:emp-risk-decomposition} and the decomposition
\(\widehat f_n-f_0=(\widehat f_n-f_{j^*})+(f_{j^*}-f_0)\),
we obtain
\begin{align*}
\EmpRisk_n(\widehat f_n,f_0)
&\le
\overline\Psi_n(\widehat f_n,\bar f)
+\EmpRisk_n(\bar f,f_0)
+2\bbE\Bigl[\Bigl\lvert \frac1n\sum_{k=0}^{n-1}(\widehat f_n-\bar f)(X_{k\Delta})I_{k\Delta}^i\Bigr\rvert\Bigr] \\
&\hspace{-1.5em}
+2\bbE\Bigl[\Bigl\lvert \frac1n\sum_{k=0}^{n-1}(f_{j^*}-f_0)(X_{k\Delta})Z_{k\Delta}^i\Bigr\rvert\Bigr]
+2\bbE\Bigl[\Bigl\lvert \frac1n\sum_{k=0}^{n-1}(\widehat f_n-f_{j^*})(X_{k\Delta})Z_{k\Delta}^{i,\prime\prime}\Bigr\rvert\Bigr] \\
&\hspace{-1.5em}
+2\bbE\Bigl[\Bigl\lvert \frac1n\sum_{k=0}^{n-1}(\widehat f_n-f_{j^*})(X_{k\Delta})Z_{k\Delta}^{i,\prime}\Bigr\rvert\Bigr]
+2\bbE\Bigl[\frac1n\sum_{k=0}^{n-1}(f_{j^*}-f_0)(X_{k\Delta})J_{k\Delta}^i\Bigr] \\
&\hspace{-1.5em}
+2\bbE\Bigl[\Bigl\lvert \frac1n\sum_{k=0}^{n-1}(\widehat f_n-f_{j^*})(X_{k\Delta})J_{k\Delta}^i\Bigr\rvert\Bigr].
\end{align*}
By Lemma~\ref{lem:diffusion-type-estimates}(iv),
\[
2\bbE\Bigl[\Bigl\lvert \frac1n\sum_{k=0}^{n-1}(\widehat f_n-\bar f)(X_{k\Delta})I_{k\Delta}^i\Bigr\rvert\Bigr]
\le
\frac{\varepsilon}{2}\EmpRisk_n(\widehat f_n,f_0)
+\varepsilon \EmpRisk_n(\bar f,f_0)
+\frac{3C_b'^2C_{\Lip}}{2\varepsilon}\Delta.
\]
Also, by Lemma~\ref{lem:diffusion-type-estimates}(v) and the inequality \(2ab\le \frac{\varepsilon}{4}b^2+\frac{4}{\varepsilon}a^2\),
\begin{align*}
&2\bbE\Bigl[\Bigl\lvert \frac1n\sum_{k=0}^{n-1}(f_{j^*}-f_0)(X_{k\Delta})Z_{k\Delta}^i\Bigr\rvert\Bigr] \\
&\le
\frac{\varepsilon}{4}\EmpRisk_n(\widehat f_n,f_0)
+\varepsilon F\delta
+\frac{\varepsilon F^2 C_\Sigma'^2 C_{\Lip}}{2C_\Sigma^2}\Delta
+\frac{16C_\Sigma^2}{\varepsilon n\Delta}(2\log N_n+\log 2).
\end{align*}
Moreover, Lemma~\ref{lem:diffusion-type-estimates}(vi) and (vii) yield
\[
2\bbE\Bigl[\Bigl\lvert \frac1n\sum_{k=0}^{n-1}(\widehat f_n-f_{j^*})(X_{k\Delta})Z_{k\Delta}^{i,\prime\prime}\Bigr\rvert\Bigr]
\le
2\delta\sqrt{\frac12 C_\Sigma'^2 C_{\Lip}},
\]
and
\[
2\bbE\Bigl[\Bigl\lvert \frac1n\sum_{k=0}^{n-1}(\widehat f_n-f_{j^*})(X_{k\Delta})Z_{k\Delta}^{i,\prime}\Bigr\rvert\Bigr]
\le
\frac{2C_Z}{\sqrt{n\Delta}}
\int_0^\delta \sqrt{\log \NN(u,\ClassF,\norminf{\cdot})}\,du.
\]
For the jump terms, Lemmas~\ref{lem:jump-net-main-term} and~\ref{lem:jump-chaining} imply
\[
2\bbE\Bigl[\frac1n\sum_{k=0}^{n-1}(f_{j^*}-f_0)(X_{k\Delta})J_{k\Delta}^i\Bigr]
\le
\frac{\varepsilon}{4}\EmpRisk_n(\widehat f_n,f_0)
+\varepsilon F\delta
+\frac{2\log N_n}{\varepsilon_J n\Delta},
\]
where \(\varepsilon_J\) is the constant from Lemma~\ref{lem:jump-net-main-term}, and
\[
2\bbE\Bigl[\Bigl\lvert \frac1n\sum_{k=0}^{n-1}(\widehat f_n-f_{j^*})(X_{k\Delta})J_{k\Delta}^i\Bigr\rvert\Bigr]
\le
\frac{2C_J}{\sqrt{n\Delta}}
\int_0^\delta \log \NN(u,\ClassF,\norminf{\cdot})\,du.
\]
Moreover, since \(\sqrt{x}\le x+1\) for all \(x\ge0\) and
\(n\Delta\ge2\), we have
\[
\begin{aligned}
\frac{2C_Z}{\sqrt{n\Delta}}
\int_0^\delta
\sqrt{\log \NN(u,\ClassF,\norminf{\cdot})}\,du
&\le
\frac{2C_Z}{\sqrt{n\Delta}}
\int_0^\delta
\log \NN(u,\ClassF,\norminf{\cdot})\,du
+
\frac{2C_Z\delta}{\sqrt{n\Delta}}  \\
&\le
\frac{2C_Z}{\sqrt{n\Delta}}
\int_0^\delta
\log \NN(u,\ClassF,\norminf{\cdot})\,du
+
2C_Z\delta .
\end{aligned}
\]
Consequently, the Brownian chaining term and the jump chaining term are
bounded together by
\[
2C_Z\delta
+
\frac{2(C_Z+C_J)}{\sqrt{n\Delta}}
\int_0^\delta
\log \NN(u,\ClassF,\norminf{\cdot})\,du .
\]
Therefore,
\begin{align*}
\EmpRisk_n(\widehat f_n,f_0)
&\le
\overline\Psi_n(\widehat f_n,\bar f)
+(1+\varepsilon)\EmpRisk_n(\bar f,f_0)
+\varepsilon\EmpRisk_n(\widehat f_n,f_0) \\
&\quad
+\left(
\frac{3C_b'^2C_{\Lip}}{2\varepsilon}
+\frac{\varepsilon F^2 C_\Sigma'^2 C_{\Lip}}{2C_\Sigma^2}
\right)\Delta
+2\varepsilon F\delta
+2\delta\sqrt{\frac12 C_\Sigma'^2 C_{\Lip}} \\
&\quad
+2C_Z\delta
+\frac{16C_\Sigma^2}{\varepsilon n\Delta}(2\log N_n+\log 2)
+\frac{2\log N_n}{\varepsilon_J n\Delta} \\
&\quad
+\frac{2(C_Z+C_J)}{\sqrt{n\Delta}}
\int_0^\delta
\log \NN(u,\ClassF,\norminf{\cdot})\,du .
\end{align*}
Since \(F\ge 1\), there exists a constant \(C_\varepsilon>0\), depending only on
\((C_b,\allowbreak C_b',\allowbreak C_\Sigma,\allowbreak C_\Sigma',
\allowbreak C_\Xi,\allowbreak C_\Xi',\allowbreak C_\nu,\allowbreak C_m,
\allowbreak\lambda,\allowbreak\varepsilon)\),
such that
\[
\EmpRisk_n(\widehat f_n,f_0)
\le
\overline\Psi_n(\widehat f_n,\bar f)
+(1+\varepsilon)\EmpRisk_n(\bar f,f_0)
+\varepsilon\EmpRisk_n(\widehat f_n,f_0)
+C_\varepsilon F^2\,\gamma^{(L)}_{\delta,n}.
\]
Hence
\[
(1-\varepsilon)\EmpRisk_n(\widehat f_n,f_0)
\le
\overline\Psi_n(\widehat f_n,\bar f)
+(1+\varepsilon)\EmpRisk_n(\bar f,f_0)
+C_\varepsilon F^2\,\gamma^{(L)}_{\delta,n}.
\]
Since
\[
\overline\Psi_n(\widehat f_n,\bar f)
=
\bbE\bigl[
\ClassQ_n(\widehat f_n)-\ClassQ_n(\bar f)
\bigr]
\le
\bbE\biggl[
\ClassQ_n(\widehat f_n)
-\inf_{f\in\ClassF}\ClassQ_n(f)
\biggr]
=
\Psi_n^{\ClassF}(\widehat f_n),
\]
and \(\EmpRisk_n(\bar f,f_0)=\Risk_n(\bar f,f_0)\) for every non-random
\(\bar f\in\ClassF\),
we conclude, after enlarging \(C_\varepsilon\) if necessary, that
\[
\EmpRisk_n(\widehat f_n,f_0)
\le
\frac{1}{1-\varepsilon}\Psi_n^{\ClassF}(\widehat f_n)
+\frac{1+\varepsilon}{1-\varepsilon}\Risk_n(\bar f,f_0)
+C_\varepsilon F^2\,\gamma^{(L)}_{\delta,n}.
\]
Taking the infimum over \(\bar f\in\ClassF\) completes the proof.
\end{proof}

\subsection{Proof of Theorem~\ref{thm:oracle}}

\begin{proof}[Proof of Theorem~\ref{thm:oracle}]
Set
\[
\ClassF:=\ClassF(L,\mathbf{p},s,F),
\qquad
\delta:=(n\Delta)^{-1},
\qquad
a_n:=\frac{\log(n\Delta)}{C_\beta\Delta},
\qquad
\mu_n:=\left\lfloor \frac{n}{2a_n}\right\rfloor,
\]
and
\[
N_n:=\NN(\delta,\ClassF,\norminf{\cdot})\vee \lceil e^{7.4}\rceil .
\]
We may assume \(\mu_n>0\). Indeed, if \(\mu_n=0\), then
\(n<2a_n\), and hence
\(\log(n\Delta)/(n\Delta)>C_\beta/2\).
Since both \(\widehat f_n\) and \(f_0\) are bounded by \(F\), we have
\[
\Risk_n(\widehat f_n,f_0)
\le 4F^2
\le
\frac{8F^2}{C_\beta}
\frac{\log(n\Delta)}{n\Delta}.
\]
This term is absorbed into the right-hand side of the desired bound
after enlarging the constant. Thus, in the sequel, we suppose that
\(\mu_n>0\).

Fix \(\tau>1\). Choose \(\varepsilon\in(0,1)\) so small that
\((1+\varepsilon)^2/(1-\varepsilon)\le \tau\).
Applying Lemmas~\ref{lem:risk-to-emp-risk} and
\ref{lem:emp-risk-bound} with this \(\varepsilon\), and using
\(\Delta\le1\) and \(F\ge1\), we obtain
\begin{align*}
\Risk_n(\widehat f_n,f_0)
&\le
\tau\left(
\Psi_n(\widehat f_n)
+
\inf_{f\in\ClassF}\Risk_n(f,f_0)
\right)  \\
&\quad
+
c_1F^2
\left(
\frac{\log N_n}{\mu_n}
+
\beta_X(a_n\Delta)
+
\frac{1}{n}
+
\delta
+
\Delta
+
\Gamma_n^{(J)}
+
\frac{\log N_n}{n\Delta}
\right),
\end{align*}
where
\[
\Gamma_n^{(J)}
:=
\frac{1}{\sqrt{n\Delta}}
\int_0^\delta \log \NN(u,\ClassF,\norminf{\cdot})\,du,
\]
and \(c_1>0\) depends only on
\((d,\allowbreak C_b,\allowbreak C_b',\allowbreak C_\Sigma,\allowbreak C_\Sigma',
\allowbreak C_\Xi,\allowbreak C_\Xi',\allowbreak C_\nu,\allowbreak C_m,
\allowbreak\lambda,\allowbreak C_\beta,\allowbreak C_\beta',\allowbreak\tau)\).

Oga and Koike~\cite[Lemma~4.13]{oga-koi24} show that, for \(0<u\le\delta\),
\[
\log \NN(u,\ClassF,\norminf{\cdot})
\le
c_2 s\bigl(L\log s+\log d+\log u^{-1}\bigr),
\]
where \(c_2>0\) is a universal constant. Since
\(u\le \delta=(n\Delta)^{-1}\le 1/2\) by
Assumption~\ref{asm:sampling}, we have
\(\log d \le (\log d/\log 2)\log u^{-1}\),
and hence
\[
\log \NN(u,\ClassF,\norminf{\cdot})
\le
c_3 s\bigl(L\log s+\log u^{-1}\bigr),
\qquad 0<u\le \delta,
\]
for a constant \(c_3>0\) depending only on \(d\).

Therefore,
\begin{align*}
\Gamma_n^{(J)}
&\le
\frac{c_3}{\sqrt{n\Delta}}
\int_0^{(n\Delta)^{-1}}
s(L\log s+\log u^{-1})\,du \\
&=
\frac{c_3s}{\sqrt{n\Delta}}
\int_{n\Delta}^{\infty}
\frac{L\log s+\log r}{r^2}\,dr \\
&\le
c_4
\frac{s(L\log s+\log(n\Delta)+1)}{(n\Delta)^{3/2}} \\
&\le
c_5
\frac{s(L\log s+\log(n\Delta))\log(n\Delta)}{n\Delta},
\end{align*}
where the last inequality follows from Assumption~\ref{asm:sampling},
since \(n\Delta\ge2\).

Moreover, Oga and Koike~\cite[Lemma~4.13]{oga-koi24} and
Assumption~\ref{asm:sampling} imply
\[
\log N_n
\le
c_6 s\bigl(L\log s+\log(n\Delta)\bigr).
\]
Therefore,
\[
\frac{\log N_n}{\mu_n}+\frac{\log N_n}{n\Delta}
\le
c_7
\frac{s(L\log s+\log(n\Delta))\log(n\Delta)}{n\Delta}.
\]
Also, Assumption~\ref{asm:mixing} yields
\[
\beta_X(a_n\Delta)
\le
C_\beta' e^{-C_\beta a_n\Delta}
=
C_\beta'(n\Delta)^{-1}.
\]
Since \(\delta=(n\Delta)^{-1}\), \(n\Delta\ge2\), and
\(s\ge2\), all the remaining lower-order terms are absorbed into
\[
\frac{s(L\log s+\log(n\Delta))\log(n\Delta)}{n\Delta}
+\Delta.
\]
Finally, note that
\(\Risk_n(f,f_0)\le \norminf{f-f_0}^2\) for every \(f\in\ClassF\).
Combining the estimates above completes the proof.
\end{proof}

\section{Proof of Theorem~\ref{thm:composition}}
\label{sec:proof-composition}

The next approximation result essentially follows from the argument used for Schmidt-Hieber~\cite[Theorem~1]{sch-hie20}; see also Oga and Koike~\cite[Appendix~A.2]{oga-koi24} for the corresponding argument in the present compositional setup.

\begin{lemma}\label{lem:composition-approx}
For any \(f\in\ClassG(q,\mathbf{d},\mathbf{t},\bm{\beta},K)\), there exist a function \(\widetilde f_n^*\in\ClassF(L,\mathbf{p},s,F)\) and a constant \(C>0\), depending only on
\((c_L^u,\allowbreak c_p,\allowbreak c_s^{\ell},\allowbreak c_s^u,
\allowbreak q,\allowbreak\mathbf{d},\allowbreak\mathbf{t},\allowbreak\bm{\beta},\allowbreak F)\),
such that
\[
\norminf{\widetilde f_n^*-f}^2 \le C\phi_n.
\]
\end{lemma}

\begin{proof}[Proof of Theorem~\ref{thm:composition}]
Assumption~\ref{asm:sampling} together with condition~(iv) yields
\[
\frac{s(L\log s+\log(n\Delta))\log(n\Delta)}{n\Delta}
\le
c\,\phi_n L\log^3(n\Delta),
\]
where \(c>0\) depends only on \((c_s^u,q,\mathbf{t},\bm{\beta})\). Applying Theorem~\ref{thm:oracle} and Lemma~\ref{lem:composition-approx} with \(f=f_0\), we obtain the claimed bound.
\end{proof}

\section{Proof of Lemma~\ref{lem:uniform-mixing-gaussian-cp}}
\label{sec:proof-mixing}

We follow the proof strategy of Oga and Koike~\cite[Lemma~3.1]{oga-koi24}.
The only additional point is the treatment of the compound-Poisson part in
the Foster--Lyapunov drift estimate.

Let \((P_b^t)_{t\ge0}\) denote the transition semigroup of
\eqref{eq:gaussian-cp-submodel}, and let \(\mathcal A_b^J\) denote its infinitesimal generator. We first verify a uniform Foster--Lyapunov
drift condition. Put \(\rho(x)=(1+\abs{x}^2)^{1/2}\), \(V(x)=\exp(\theta\rho(x))\), where \(\theta\in(0,r)\) will be chosen sufficiently small. Specifically,
we show that there exist constants \(a,C_0>0\), independent of
\(b\in\ClassB_0^J(d,r,K)\), such that
\[
    \mathcal A_b^J V(x)\le -aV(x)+C_0,
    \qquad x\in\bbR^d.
\]
Since \(\rho(x)\le1+\abs{x}\) and \(\theta\le r\),
\begin{equation}\label{eq:levy-initial-V-moment}
    \bbE[V(\eta)]
    \le
    e^\theta\bbE[e^{r\abs{\eta}}]
    <\infty .
\end{equation}
Set \(K_d:=\sqrt d\,K\). Then
\[
    \sup_{b\in\ClassB_0^J(d,r,K)}
    \sup_{x\in\bbR^d}\abs{b(x)}
    \le K_d .
\]

For \(f\in C^2(\bbR^d)\) for which the integral below is finite, the generator \(\mathcal A_b^J\) is given by
\[
    \mathcal A_b^J f(x)
    =
    \langle b(x),\nabla f(x)\rangle
    +\frac12\Delta f(x)
    +
    \iota\int_{\bbR^d}
    \{f(x+z)-f(x)\}\varphi_d(z)\,dz .
\]
The diffusion part is handled in the same spirit as in
Oga and Koike~\cite[Lemma~3.1]{oga-koi24}, although we use the smooth
Lyapunov function \(V(x)=\exp(\theta\rho(x))\). Throughout this paragraph,
\(C\) denotes a positive constant independent of \(b\), \(x\), and
\(\theta\in(0,r\wedge1]\), whose value may change from line to line. Direct differentiation gives
\begin{equation}\label{eq:lyapunov-derivative-bounds}
    \nabla V(x)=\theta V(x)\frac{x}{\rho(x)},
    \qquad
    \norm{\nabla^2V(x)}
    \le
    C\left(\theta^2+\frac{\theta}{\rho(x)}\right)V(x).
\end{equation}
Therefore, using
\(\langle b(x),x\rangle\le -r\abs{x}\) for \(\abs{x}>2d\), we obtain
\begin{equation}\label{eq:diffusion-lyapunov-estimate}
\begin{aligned}
    \langle b(x),\nabla V(x)\rangle+\frac12\Delta V(x)
    &\le
    -\theta r\frac{\abs{x}}{\rho(x)}V(x)
    +C\left(\theta^2+\frac{\theta}{\rho(x)}\right)V(x)  \\
    &\le
    -\frac34\theta rV(x)
    +C\left(\theta^2+\frac{\theta}{\rho(x)}\right)V(x)
\end{aligned}
\end{equation}
for all sufficiently large \(\abs{x}\), uniformly over
\(b\in\ClassB_0^J(d,r,K)\).

It remains to estimate the jump part. Since the Gaussian jump distribution is
centered,
\[
    \int_{\bbR^d}z\,\varphi_d(z)\,dz=0,
\]
we may subtract the linear term:
\[
\begin{aligned}
&\int_{\bbR^d}
\{V(x+z)-V(x)\}\varphi_d(z)\,dz     \\
&\quad =
\int_{\bbR^d}
\{V(x+z)-V(x)-\langle\nabla V(x),z\rangle\}
\varphi_d(z)\,dz .
\end{aligned}
\]
By Taylor's formula, for \(\abs{z}\le \rho(x)/2\),
\[
\begin{aligned}
&\abs{V(x+z)-V(x)-\langle \nabla V(x),z\rangle} \\
&\qquad\le
\frac{\abs{z}^2}{2}
\sup_{0\le u\le1}\norm{\nabla^2V(x+uz)}
\le
C\left(\theta^2+\frac{\theta}{\rho(x)}\right)
V(x)\abs{z}^2e^{\theta\abs{z}}.
\end{aligned}
\]
where we used \eqref{eq:lyapunov-derivative-bounds} and
\(\rho(x+uz)\ge \rho(x)/2\).
Since
\(\int_{\bbR^d}\abs{z}^2e^{\theta\abs{z}}\varphi_d(z)\,dz<\infty\),
the contribution from \(\{\abs{z}\le\rho(x)/2\}\) is therefore bounded by
\[
    C\left(\theta^2+\frac{\theta}{\rho(x)}\right)V(x).
\]
On the complement, using again
\(V(x+z)\le V(x)e^{\theta\abs{z}}\) and
\(\abs{\nabla V(x)}\le \theta V(x)\), the contribution is bounded by
\[
    V(x)
    \int_{\{\abs{z}>\rho(x)/2\}}
    \{e^{\theta\abs{z}}+1+\theta\abs{z}\}\varphi_d(z)\,dz
    =\varepsilon_\theta(x)V(x),
\]
where \(\varepsilon_\theta(x)\to0\) as \(\abs{x}\to\infty\), by the Gaussian
exponential moments. Hence
\begin{equation}
\label{eq:jump-lyapunov-estimate}
\begin{aligned}
&\abs{
\int_{\bbR^d}
\{V(x+z)-V(x)-\langle\nabla V(x),z\rangle\}
\varphi_d(z)\,dz
}        \\
&\qquad\le
C\left(\theta^2+\frac{\theta}{\rho(x)}+\varepsilon_\theta(x)\right)V(x).
\end{aligned}
\end{equation}
This is the only estimate in the Foster--Lyapunov argument which differs from
the diffusion case.

Combining \eqref{eq:diffusion-lyapunov-estimate} and
\eqref{eq:jump-lyapunov-estimate}, and absorbing the factor \(\iota\)
into \(C\), we have, for all sufficiently large \(\abs{x}\),
\[
    \mathcal A_b^JV(x)
    \le
    \left\{
        -\frac34\theta r
        +C\left(\theta^2+\frac{\theta}{\rho(x)}+\varepsilon_\theta(x)\right)
    \right\}V(x).
\]
First choose \(\theta\in(0,r\wedge1]\) so small that
\(C\theta^2\le \theta r/8\). Then choose \(R_1>2d\) so large that
\[
    C\left(\frac{\theta}{\rho(x)}+\varepsilon_\theta(x)\right)
    \le \frac{\theta r}{8},
    \qquad \abs{x}\ge R_1 .
\]
Thus, with \(a:=\theta r/2\),
\[
    \mathcal A_b^JV(x)\le -aV(x),
    \qquad \abs{x}\ge R_1,
\]
uniformly over \(b\in\ClassB_0^J(d,r,K)\).

For \(\abs{x}\le R_1\), set
\[
    G_\theta(x)
    :=
    \iota\int_{\bbR^d}
    \abs{
        V(x+z)-V(x)-\langle\nabla V(x),z\rangle
    }
    \varphi_d(z)\,dz .
\]
Define
\[
    C_0
    :=
    1+\sup_{\abs{x}\le R_1}
    \left\{
        aV(x)
        +K_d\abs{\nabla V(x)}
        +\frac12\abs{\Delta V(x)}
        +G_\theta(x)
    \right\}.
\]
Since \(V\), \(\nabla V\), and \(\nabla^2V\) are bounded on
\(\{\abs{x}\le R_1\}\), and since the Gaussian distribution has exponential
moments of every order, \(C_0<\infty\). Moreover, \(C_0\) depends only on
\((d,r,K,\iota)\). Combining the compact and non-compact estimates yields
\begin{equation}
\label{eq:levy-generator-drift}
    \mathcal A_b^J V(x)\le -aV(x)+C_0,
    \qquad x\in\bbR^d,
\end{equation}
uniformly over \(b\in\ClassB_0^J(d,r,K)\).

We next justify the use of the unbounded Lyapunov function \(V\). The process
\(X\) is a time-homogeneous Markov process with c\`adl\`ag trajectories, so
the extended-generator framework of Kulik~\cite[Section~3.1]{kul18}
applies. By the It\^o formula for jump diffusions,
\[
    V(X_t)-V(X_0)-\int_0^t \mathcal A_b^J V(X_s)\,ds
\]
is a local martingale. The nonlocal term is well-defined since
\[
    V(x+z)\le V(x)e^{\theta\abs{z}}
    \quad\text{and}\quad
    \int_{\bbR^d}e^{\theta\abs{z}}\varphi_d(z)\,dz<\infty .
\]
Thus \(V\in\operatorname{Dom}(\mathcal A_b^J)\) in the sense of
Kulik~\cite[Definition~3.2.1]{kul18}. Applying
Kulik~\cite[Theorem~3.2.3 and Example~3.2.5]{kul18} to
\eqref{eq:levy-generator-drift}, with \(\phi(v)=av\), yields
\begin{equation}
\label{eq:levy-lyapunov-semigroup}
    P_b^tV(x)
    \le
    e^{-at}V(x)+\frac{C_0}{a}(1-e^{-at}),
    \qquad t>0 .
\end{equation}
In particular, with \(\gamma:=e^{-a}\) and
\(K_1:=(C_0/a)(1-e^{-a})\),
we have
\begin{equation}
\label{eq:levy-drift-skeleton}
    P_b^1V(x)\le \gamma V(x)+K_1,
    \qquad x\in\bbR^d,
\end{equation}
uniformly over \(b\in\ClassB_0^J(d,r,K)\).

Next we verify the minorization condition. Let \((Q_b^t)_{t\ge0}\) be the
transition semigroup of the diffusion without jumps,
\[
    dY_t=b(Y_t)\,dt+dW_t .
\]
Since the compound-Poisson process has no jump on \([0,1]\) with probability
\(e^{-\iota}\), for every Borel set \(A\subset\bbR^d\),
\[
    P_b^1(x,A)\ge e^{-\iota}Q_b^1(x,A).
\]
Let \(q_b^1(x,y)\) be the transition density of \(Q_b^1\). By the lower bound
of Qian, Russo and Zheng~\cite{qia-etal03}, which is also used by
Oga and Koike~\cite[Lemma~3.1]{oga-koi24}, there exists a constant
\(\zeta_d>0\), depending only on \(d\), such that
\[
    q_b^1(x,y)
    \ge
    \frac{1}{(2\pi)^{d/2}}
    \exp\left\{
        -\frac{(\abs{x-y}+\norm{b})^2}{2}
        -\norm{b}(\zeta_d+\abs{x-y})
    \right\},
\]
where \(\norm{b}:=\sup_{x\in\bbR^d}\abs{b(x)}\). Since \(\norm{b}\le K_d\), this
implies
\[
    q_b^1(x,y)
    \ge
    2^{-3d/2}
    \exp\{-2\abs{x}^2-dK^2(3+\zeta_d)\}\psi(y),
\]
where
\(\psi(y)=(2/\pi)^{d/2}e^{-2\abs{y}^2}\).
Set
\begin{equation}\label{eq:levy-small-set}
    R:=\frac{3C_0}{a},
    \qquad
    \mathcal C:=\{x\in\bbR^d:V(x)\le R\}.
\end{equation}
Since \(\mathcal C\) is compact, there exists
\(\alpha_{\mathrm{diff}}>0\), depending only on \((d,r,K,\iota)\), such that
\[
    Q_b^1(x,dy)\ge \alpha_{\mathrm{diff}}\psi(y)\,dy,
    \qquad x\in\mathcal C.
\]
Therefore
\begin{equation}
\label{eq:levy-minorization}
    P_b^1(x,dy)\ge \alpha_J\psi(y)\,dy,
    \qquad x\in\mathcal C,
\end{equation}
where \(\alpha_J:=e^{-\iota}\alpha_{\mathrm{diff}}>0\). This is the jump-diffusion analogue of the minorization condition used in
Oga and Koike~\cite[Section~6]{oga-koi24}.

The rest of the argument is identical to the Harris-type contraction argument
in Oga and Koike~\cite[Section~6]{oga-koi24}. The drift condition
\eqref{eq:levy-drift-skeleton} and the minorization condition
\eqref{eq:levy-minorization} imply that there exist constants
\(\beta>0\) and \(\bar\alpha\in(0,1)\), depending only on
\((d,r,K,\iota)\), such that
\begin{equation}
\label{eq:levy-weighted-contraction}
    \rho_\beta(\mu_1P_b^n,\mu_2P_b^n)
    \le
    \bar\alpha^n\rho_\beta(\mu_1,\mu_2),
    \qquad n=1,2,\ldots,
\end{equation}
for all probability measures \(\mu_1,\mu_2\) with finite \(V\)-moments, where
\[
    \rho_\beta(\mu_1,\mu_2)
    :=
    \int_{\bbR^d}(1+\beta V(x))\abs{\mu_1-\mu_2}(dx).
\]
In particular, the skeleton chain admits a unique invariant distribution
\(\Pi_b\). Moreover, integrating
\eqref{eq:levy-lyapunov-semigroup} with respect to \(\Pi_b\) gives
\begin{equation}
\label{eq:levy-invariant-moment}
    \int_{\bbR^d}V(x)\Pi_b(dx)\le \frac{C_0}{a}.
\end{equation}

Finally, by the Markov representation of the \(\beta\)-mixing coefficient,
\[
    \beta_X(t)
    =
    \sup_{s\ge0}
    \int_{\bbR^d}
    \norm{\delta_xP_b^t-\Law_\eta P_b^{s+t}}\,
    \Law_\eta P_b^s(dx).
\]
Using \eqref{eq:levy-weighted-contraction} for the integer part of \(t\), and
using \eqref{eq:levy-lyapunov-semigroup} to bound the \(V\)-moment of
\(\Law_\eta P_b^s\) uniformly in \(s\), we obtain
\[
    \beta_X(t)
    \le
    C_\beta'\exp(-C_\beta t),
    \qquad t\ge0.
\]
Here \(C_\beta>0\) depends only on \((d,r,K,\iota)\), and \(C_\beta'\)
depends only on \((d,r,K,\iota)\) and \(\bbE[V(\eta)]\). Since
\(\theta\le r\), \(\bbE[V(\eta)]\le e^\theta
\bbE[e^{r\abs{\eta}}]<\infty\). This proves the desired uniform exponential
\(\beta\)-mixing bound.

\section{Proof of Theorem~\ref{thm:minimax-lower-gaussian-cp}}
\label{sec:proof-minimax}

Throughout the rest of this section, we write
\[
    \ClassB^J
    :=
    \ClassB_i^J(r,q,\mathbf d,\mathbf t,\bm\beta,K),
    \qquad
    \norm{g}_2
    :=
    \norm{g}_{L^2([0,1]^d)} .
\]
The next lemma reduces the lower bound for the generalization error to a
standard \(L^2([0,1]^d)\)-minimax lower bound. 

\begin{lemma}
\label{lem:levy-minimax-risk-reduction}
Assume that \(\bbE[\exp(r\abs{\eta})]<\infty\), \(\Delta\le1\), and
\(n\Delta\to\infty\). Then there exists a constant \(c>0\), depending only on
\((d,r,K,\iota)\), such that
\[
\liminf_{n\to\infty}
\inf_{\widehat f_n}
\sup_{b\in\ClassB^J}
\phi_n^{-1}
\Risk_n(\widehat f_n,f_0)
\ge
c
\liminf_{n\to\infty}
\inf_{\widehat f_n}
\sup_{b\in\ClassB^J}
\phi_n^{-1}
\bbE\!\left[
    \norm{\widehat f_n-f_0}_2^2
\right],
\]
where the infima are taken over all estimators.
\end{lemma}

\begin{proof}
The proof follows the reduction argument of
Oga and Koike~\cite[Lemma~7.1]{oga-koi24}. We only verify that the
ergodic and invariant-measure bounds required in the argument are uniform
over the present L\'evy submodel.

Throughout the proof, we use the notation introduced in
Section~\ref{sec:proof-mixing}. Let \(\Pi_b\) be the
invariant distribution of the skeleton chain obtained there. For
\(t\ge0\), write \(t=m+s\), where \(m=\lfloor t\rfloor\) and
\(s\in[0,1)\). By the semigroup property, the invariance of \(\Pi_b\) for the skeleton
chain, and \eqref{eq:levy-weighted-contraction},
\[
\begin{aligned}
\rho_\beta(\delta_xP_b^t,\Pi_b)
&=
\rho_\beta\bigl((\delta_xP_b^s)P_b^m,\Pi_bP_b^m\bigr) \\
&\le
\bar\alpha^m\rho_\beta(\delta_xP_b^s,\Pi_b) \\
&\le
\bar\alpha^m
\left\{
2+\beta P_b^sV(x)
+\beta\int_{\bbR^d}V(y)\Pi_b(dy)
\right\} \\
&\le
C\bar\alpha^m(1+V(x)),
\end{aligned}
\]
where the last inequality follows from
\eqref{eq:levy-lyapunov-semigroup} and
\eqref{eq:levy-invariant-moment}. Hence there exist constants \(C,c>0\),
independent of \(b\in\ClassB^J\), such that
\begin{equation}\label{eq:levy-tv-ergodicity}
    \norm{\delta_xP_b^t-\Pi_b}_{\mathrm{TV}}
    \le
    Ce^{-ct}(1+V(x)),
    \qquad x\in\bbR^d,\quad t\ge0.
\end{equation}

We also need a lower bound for \(\Pi_b\) on \([0,1]^d\).
By the invariance of \(\Pi_b\) and the minorization condition
\eqref{eq:levy-minorization},
\[
    \Pi_b(dy)
    =
    \int_{\bbR^d}P_b^1(x,dy)\Pi_b(dx)
    \ge
    \alpha_J\Pi_b(\mathcal C)\psi(y)\,dy .
\]
Moreover, by \eqref{eq:levy-invariant-moment},
\eqref{eq:levy-small-set}, and Markov's inequality,
\[
    1-\Pi_b(\mathcal C)
    =
    \Pi_b(V>R)
    \le
    R^{-1}\int_{\bbR^d}V(x)\Pi_b(dx)
    \le
    \frac13 .
\]
Consequently, there exists a constant \(c_0>0\), depending only on
\((d,r,K,\iota)\), such that
\begin{equation}
\label{eq:levy-invariant-lower-bound}
    \Pi_b(dy)
    \ge
    c_0\one_{[0,1]^d}(y)\,dy,
    \qquad b\in\ClassB^J .
\end{equation}

Fix an arbitrary estimator \(\widehat f_n\), and define its clipped version by
\[
    \widetilde f_n
    :=
    \bigl(\{(-K)\vee \widehat f_n\}\wedge K\bigr)
    \one_{[0,1]^d}.
\]
Since \(\norminf{f_0}\le K\) and
\(\supp(f_0)\subset[0,1]^d\), we have
\(\abs{\widetilde f_n-f_0}\le \abs{\widehat f_n-f_0}\).
Hence
\begin{equation}\label{eq:levy-risk-clipping}
    \Risk_n(\widehat f_n,f_0)
    \ge
    \Risk_n(\widetilde f_n,f_0).
\end{equation}

Using the preceding total variation bound and
\(\abs{\widetilde f_n-f_0}\le2K\), we obtain, uniformly in
\(b\in\ClassB^J\),
\begin{equation}\label{eq:levy-risk-stationary-comparison}
\begin{aligned}
&\abs{
\Risk_n(\widetilde f_n,f_0)
-
\bbE\left[
    \int_{\bbR^d}
    \abs{\widetilde f_n(y)-f_0(y)}^2\,\Pi_b(dy)
\right]
}                                                    \\
&\qquad\le
4K^2
\frac1n\sum_{k=0}^{n-1}
\int_{\bbR^d}
\norm{\delta_xP_b^{k\Delta}-\Pi_b}_{\mathrm{TV}}\,
\Law_\eta(dx)
\le
\frac{C}{n\Delta}.
\end{aligned}
\end{equation}
The last inequality follows from
\eqref{eq:levy-tv-ergodicity},
\eqref{eq:levy-initial-V-moment}, and
\[
    \frac1n\sum_{k=0}^{n-1}e^{-ck\Delta}
    \le
    \frac{C}{n\Delta},
    \qquad \Delta\le1.
\]
The constant \(C\) may depend on
\(\bbE[V(\eta)]\), but is independent of \(b\), \(n\), and \(\Delta\).

By \eqref{eq:levy-invariant-lower-bound},
\begin{equation}\label{eq:levy-stationary-L2-lower-bound}
    \bbE\left[
    \int_{\bbR^d}
    \abs{\widetilde f_n(y)-f_0(y)}^2\,\Pi_b(dy)
    \right]
    \ge
    c_0
    \bbE\left[
        \norm{\widetilde f_n-f_0}_2^2
    \right].
\end{equation}
Combining
\eqref{eq:levy-risk-clipping},
\eqref{eq:levy-risk-stationary-comparison}, and
\eqref{eq:levy-stationary-L2-lower-bound}, we obtain
\[
    \Risk_n(\widehat f_n,f_0)
    \ge
    c_0
    \bbE\left[
        \norm{\widetilde f_n-f_0}_2^2
    \right]
    -
    \frac{C}{n\Delta}.
\]Taking the supremum over \(b\in\ClassB^J\) and then the infimum over all
estimators yields
\[
\inf_{\widehat f_n}
\sup_{b\in\ClassB^J}
\Risk_n(\widehat f_n,f_0)
\ge
c_0
\inf_{\widehat f_n}
\sup_{b\in\ClassB^J}
\bbE\left[
    \norm{\widehat f_n-f_0}_2^2
\right]
-
\frac{C}{n\Delta},
\]
because \(\widetilde f_n\) is itself an estimator.

Finally, since
\(\phi_n^{-1}=o(n\Delta)\),
multiplying both sides by \(\phi_n^{-1}\) and taking the lower limit gives the
claim.
\end{proof}

By Lemma~\ref{lem:levy-minimax-risk-reduction}, it remains to derive a lower
bound for the \(L^2([0,1]^d)\)-estimation error over
\(\ClassGb(q,\mathbf d,\mathbf t,\boldsymbol\beta,K)\).
The following lemma provides a finite packing of the compositional class
together with compactly supported, globally Lipschitz extensions needed
for the drift construction. Its proof is given in
\ref{app:packing-compositional-class}.

Set \(R_0:=(\sqrt d+2d)/2\).

\begin{lemma}
\label{lem:packing-compositional-class}
Assume that \(\beta^{**}\ge1\) and
\(t_j\le\min\{d_0,\ldots,d_{j-1}\}\) for \(j=1,\ldots,q\).
There exist \(K_{\rm pack},C_{\rm pack}>0\), depending only on
\((q,\mathbf d,\mathbf t,\boldsymbol\beta)\), such that for every \(T>0\)
there are \(\kappa_T,N_T>0\), depending only on \(T\) and these class
parameters, for which the following holds whenever \(n\Delta\ge N_T\).
There exist an integer \(M\ge2\), functions
\[
f_{(0)},\ldots,f_{(M)}
\in
\ClassGb
(q,\mathbf d,\mathbf t,\boldsymbol\beta,K_{\rm pack}),
\]
and extensions
\(\widetilde f_{(0)},\ldots,\widetilde f_{(M)}:\bbR^d\to\bbR\)
satisfying:
\begin{enumerate}
    \item[\emph{(i)}] For all \(0\le j<k\le M\),
    \[
        n\Delta\,
        \norm{\widetilde f_{(j)}-\widetilde f_{(k)}
        }_{L^2(\bbR^d)}^2
        \le
        \frac{\log M}{9T},
    \]
    and
    \[
        \norm{f_{(j)}-f_{(k)}}_{L^2([0,1]^d)}^2
        \ge
        \kappa_T\phi_n .
    \]

    \item[\emph{(ii)}] For every \(j=0,1,\ldots,M\),
    \[
        f_{(j)}
        =
        \widetilde f_{(j)}\one_{[0,1]^d},
    \]
    \[
        \supp(\widetilde f_{(j)})
        \subset
        \{x\in\bbR^d:\abs{x}\le R_0\},
    \]
    and
    \[
        \abs{\widetilde f_{(j)}(x)-\widetilde f_{(j)}(y)}
        \le
        K_{\rm pack}\abs{x-y},
        \qquad
        x,y\in\bbR^d .
    \]

    \item[\emph{(iii)}] For every \(j=0,1,\ldots,M\),
    \[
        \norminf{\widetilde f_{(j)}}
        \le C_{\rm pack}.
    \]
\end{enumerate}
\end{lemma}

We next embed this packing into a family of admissible drift coefficients.
For a bounded globally Lipschitz drift \(b\), let \(P_{b,n}\) denote the law of
\(X^n=(X_t)_{0\le t\le n\Delta}\) on
\(D([0,n\Delta],\bbR^d)\) under \eqref{eq:gaussian-cp-submodel}, and let
\(E_{b,n}\) denote the corresponding expectation.

Set
\[
    R
    :=
    r\vee C_{\rm pack},
    \qquad
    t_0
    :=
    \frac{\log2}{36R^2}.
\]
Let
\[
    v_d
    :=
    \lambda_d\bigl(\{z\in\bbR^d: \abs{z}\le1\}\bigr)
    =
    \frac{\pi^{d/2}}{\Gamma(d/2+1)}
\]
be the volume of the \(d\)-dimensional unit ball, and define
\[
    \Gamma
    :=
    \frac{2^d}{v_d}
    \left(
        R+\frac{1}{\sqrt{2\pi t_0}}
    \right)^d .
\]
Let \(\psi\in C^\infty(\bbR;[0,1])\) be such that
\[
    \psi(u)=0\quad\text{for } \abs{u}\le R_0,
    \qquad
    \psi(u)=1\quad\text{for } \abs{u}>2d .
\]
Define
\[
    D(x)
    :=
    \begin{cases}
        r\dfrac{x}{\abs{x}}\psi(\abs{x}),&x\ne0,\\[2mm]
        0,&x=0,
    \end{cases}
\]
and denote its Lipschitz constant by \(L_D<\infty\).
Choose
\[
    K_0\ge R\vee(K_{\rm pack}+L_D)
\]
and fix \(K\ge K_0\).

Apply Lemma~\ref{lem:packing-compositional-class} with \(T=\Gamma\), and take
\(n\) sufficiently large that
\[
    n\Delta\ge N_\Gamma\vee t_0 .
\]
This yields \(M\ge2\), functions
\(f_{(0)},\ldots,f_{(M)}\), and extensions
\(\widetilde f_{(0)},\ldots,\widetilde f_{(M)}\) satisfying the properties of
Lemma~\ref{lem:packing-compositional-class}.

For \(j=0,\ldots,M\), define
\begin{equation}
\label{eq:levy-fano-drift-family}
    b_j(x)
    :=
    \widetilde f_{(j)}(x)\mathbf e_i-D(x),
    \qquad x\in\bbR^d,
\end{equation}
where \(\mathbf e_i\) is the \(i\)-th coordinate vector.

Since
\(\supp(\widetilde f_{(j)})
\subset\{\abs{x}\le R_0\}\)
and \(D=0\) on \(\{\abs{x}\le R_0\}\),
\begin{equation}
\label{eq:levy-fano-drift-bound}
    \abs{b_j(x)}
    \le
    \max\{\norminf{\widetilde f_{(j)}},r\}
    \le R,
    \qquad x\in\bbR^d.
\end{equation}
Moreover, for all \(x,y\in\bbR^d\),
\begin{equation}
\label{eq:levy-fano-drift-lipschitz}
\begin{aligned}
    \abs{b_j(x)-b_j(y)}
    &\le
    \abs{\widetilde f_{(j)}(x)-\widetilde f_{(j)}(y)}
    +
    \abs{D(x)-D(y)} \\
    &\le
    (K_{\rm pack}+L_D)\abs{x-y}
    \le
    K\abs{x-y}.
\end{aligned}
\end{equation}
Thus \(b_0,\ldots,b_M\) are bounded and globally Lipschitz.

Write
\[
    P_j:=P_{b_j,n},
    \qquad
    E_{j,n}:=E_{b_j,n},
    \qquad
    j=0,\ldots,M,
\]
and, for \(j=1,\ldots,M\), set
\[
    \theta_j(x)
    :=
    b_j(x)-b_0(x)
    =
    \bigl(
        \widetilde f_{(j)}(x)-\widetilde f_{(0)}(x)
    \bigr)\mathbf e_i.
\]

\begin{lemma}
\label{lem:levy-girsanov-kl}
For every \(j=1,\ldots,M\),
\[
    P_j\sim P_0,
\]
and
\begin{equation}
\label{eq:levy-kl-identity-final}
    \int
    \log\left(\frac{dP_j}{dP_0}\right)dP_j
    =
    \frac12
    E_{j,n}\left[
        \int_0^{n\Delta}
        \abs{\widetilde f_{(j)}(X_u)-\widetilde f_{(0)}(X_u)}^2\,du
    \right].
\end{equation}
\end{lemma}

\begin{proof}
We apply S{\o}rensen~\cite[Corollary~2.2]{sor91} with \(T=n\Delta\).
For \(x\in\bbR^d\), let \(P_j^x\) denote the law of \(X\) on
\(D([0,T],\bbR^d)\) conditional on \(X_0=x\).
The assumptions of the corollary are satisfied since \(b_j\) is bounded,
the diffusion covariance is \(I_d\), and, for every \(j\), the jump mechanism
is the same compound-Poisson mechanism with intensity \(\iota\), jump-size
density \(\varphi_d\), and jump transformation \(z\mapsto z\).
Hence the jump-density ratio in S{\o}rensen's notation is identically one.
Moreover,
\[
    \int_0^T \abs{b_j(X_{u-})-b_0(X_{u-})}^2\,du<\infty
\]
almost surely, since \(b_j-b_0=\theta_j\) is bounded.

Hence \(P_j^x\sim P_0^x\), and
\[
    \frac{dP_j^x}{dP_0^x}
    =
    \exp\left\{
        \int_0^T \theta_j(X_{u-})^\top dM_u^{0,c}
        -
        \frac12\int_0^T \abs{\theta_j(X_{u-})}^2\,du
    \right\},
\]
where \(M^{0,c}\) is the continuous martingale part under \(P_0^x\).
Since the initial law is common to all \(j\), integration with respect to
\(\Law(\eta)\) gives the same identity for \(P_j\) and \(P_0\), and
in particular \(P_j\sim P_0\).

Under \(P_j\),
\[
    M_t^{j,c}
    =
    M_t^{0,c}
    -
    \int_0^t \theta_j(X_{u-})\,du,
\]
so
\[
    \log\left(\frac{dP_j}{dP_0}\right)
    =
    \int_0^T \theta_j(X_{u-})^\top dM_u^{j,c}
    +
    \frac12\int_0^T \abs{\theta_j(X_{u-})}^2\,du .
\]
The stochastic integral has mean zero because \(\theta_j\) is bounded.
Therefore,
\[
    \int
    \log\left(\frac{dP_j}{dP_0}\right)dP_j
    =
    \frac12
    E_{j,n}\left[
        \int_0^T \abs{\theta_j(X_u)}^2\,du
    \right].
\]
Finally,
\(\abs{\theta_j(x)}^2 = \abs{\widetilde f_{(j)}(x)-\widetilde f_{(0)}(x)}^2\),
which gives \eqref{eq:levy-kl-identity-final}.
\end{proof}

\begin{lemma}
\label{lem:levy-occupation-bound}
For every \(j=1,\ldots,M\),
\[
\begin{aligned}
& E_{j,n}\left[
    \int_0^{n\Delta}
    \abs{\widetilde f_{(j)}(X_u)-\widetilde f_{(0)}(X_u)}^2\,du
\right] \\
&\qquad\le
4R^2t_0
+
\Gamma n\Delta
\norm{\widetilde f_{(j)}-\widetilde f_{(0)}
}_{L^2(\bbR^d)}^2 .
\end{aligned}
\]
\end{lemma}

\begin{proof}
Conditionally on the compound-Poisson path \(J\), define
\[
    Y_u:=X_u-J_u,\qquad 0\le u\le n\Delta .
\]
Then \(Y\) is continuous and satisfies
\[
    dY_u=b_j(Y_u+J_{u-})\,du+dW_u .
\]
For each fixed path of \(J\), define
\[
    a_J(u,z):=b_j(z+J_{u-}) .
\]
Using \(J_{u-}\) makes \(a_J\) predictable in time, and replacing \(J_{u-}\)
by \(J_u\) would not change the drift integral because \(J_{u-}=J_u\) for
Lebesgue-a.e. \(u\). Moreover,
\[
    \abs{a_J(u,z)}\le R,
    \qquad u\ge0,\ z\in\bbR^d,
\]
uniformly in the path of \(J\). Hence, by the density bounds of Ba\~nos and
Kr\"uhner
\cite[Theorem~2.1, Theorem~2.2 and Corollary~2.3]{ban-kru16},
the conditional transition density \(q_{j,J}^t(x,\cdot)\) of \(Y_t\), given
\(Y_0=x\) and \(J\), satisfies
\[
    q_{j,J}^t(x,z)
    \le
    \frac{2^d}{v_d}
    \left(
        R+\frac{1}{\sqrt{2\pi t}}
    \right)^d,
    \qquad t>0,\quad z\in\bbR^d.
\]
In particular, for \(t\ge t_0\),
\[
    q_{j,J}^t(x,z)\le \Gamma .
\]
Since \(X_t=Y_t+J_t\), for every Borel set \(A\subset\bbR^d\),
\[
\begin{aligned}
    P_j^x(X_t\in A\mid J)
    &=
    \int_A q_{j,J}^t(x,y-J_t)\,dy,
\end{aligned}
\]
where \(P_j^x\) denotes the law starting from \(X_0=x\). Taking expectation
with respect to \(J\), we obtain a version of the transition density of \(X_t\)
under \(P_j^x\):
\[
    p_j^t(x,y)
    =
    E_J\left[
        q_{j,J}^t(x,y-J_t)
    \right].
\]
Therefore,
\begin{equation}
\label{eq:levy-transition-upper-minimax}
    p_j^t(x,y)\le \Gamma,
    \qquad
    t\ge t_0,\quad x,y\in\bbR^d .
\end{equation}

Using
\(\norminf{\widetilde f_{(j)}-\widetilde f_{(0)}} \le 2R\),
\eqref{eq:levy-transition-upper-minimax}, and Fubini's theorem, we obtain
\[
\begin{aligned}
& E_{j,n}\left[
    \int_0^{n\Delta}
    \abs{\widetilde f_{(j)}(X_u)-\widetilde f_{(0)}(X_u)}^2\,du
\right] \\
&\qquad\le
4R^2t_0
+
\int_{t_0}^{n\Delta}
\int_{\bbR^d}
\int_{\bbR^d}
\abs{\widetilde f_{(j)}(y)-\widetilde f_{(0)}(y)}^2
p_j^u(x,y)\,dy\,\Law_\eta(dx)\,du \\
&\qquad\le
4R^2t_0
+
\Gamma n\Delta
\norm{\widetilde f_{(j)}-\widetilde f_{(0)}
}_{L^2(\bbR^d)}^2 .
\end{aligned}
\]
This proves the lemma.
\end{proof}

\begin{lemma}
\label{lem:levy-fano-family}
For all sufficiently large \(n\), the functions
\(b_0,b_1,\ldots,b_M\) defined in
\eqref{eq:levy-fano-drift-family} satisfy the following properties:
\begin{enumerate}
    \item[\emph{(i)}]
    \[
        b_0,b_1,\ldots,b_M\in\ClassB^J.
    \]

    \item[\emph{(ii)}]
    There exists a constant \(\kappa>0\), independent of \(n\), such that
    \begin{equation}
    \label{eq:levy-fano-separation}
        \norm{b_j^i-b_k^i}_2^2
        \ge
        \kappa\phi_n,
        \qquad 0\le j<k\le M.
    \end{equation}

    \item[\emph{(iii)}]
    \begin{equation}
    \label{eq:levy-fano-ac}
        P_j\ll P_0,
        \qquad j=1,\ldots,M.
    \end{equation}

    \item[\emph{(iv)}]
    \begin{equation}
    \label{eq:levy-fano-kl}
        \frac1M\sum_{j=1}^M
        \int
        \log\left(\frac{dP_j}{dP_0}\right)dP_j
        \le
        \frac19\log M .
    \end{equation}
\end{enumerate}
\end{lemma}

\begin{proof}
We first verify \emph{(i)} and \emph{(ii)}.
The bounds \eqref{eq:levy-fano-drift-bound} and
\eqref{eq:levy-fano-drift-lipschitz} already give the required boundedness
and global Lipschitz continuity.

For \(\abs{x}>2d\),
\(\widetilde f_{(j)}(x)=0\) and \(\psi(\abs{x})=1\), so
\(\langle b_j(x),x\rangle=-r\abs{x}\).
Moreover, since
\([0,1]^d\subset\{\abs{x}<R_0\}\),
we have \(D=0\) on \([0,1]^d\), and property \emph{(ii)} of
Lemma~\ref{lem:packing-compositional-class} yields
\[
    b_j^i\one_{[0,1]^d}
    =
    f_{(j)}
    \in
    \ClassGb
    (q,\mathbf d,\mathbf t,\bm\beta,K_{\rm pack})
    \subset
    \ClassGb
    (q,\mathbf d,\mathbf t,\bm\beta,K).
\]
Together with \eqref{eq:levy-fano-drift-bound} and
\eqref{eq:levy-fano-drift-lipschitz}, this shows that
\(b_0,\ldots,b_M\in\ClassB^J\),
proving \emph{(i)}.

Moreover, \(b_j^i-b_k^i=f_{(j)}-f_{(k)}\) on \([0,1]^d\),
and therefore property \emph{(i)} of
Lemma~\ref{lem:packing-compositional-class} yields
\[
    \norm{b_j^i-b_k^i}_2^2
    \ge
    \kappa_\Gamma\phi_n,
    \qquad 0\le j<k\le M.
\]
Thus \eqref{eq:levy-fano-separation} holds with
\(\kappa=\kappa_\Gamma>0\), proving \emph{(ii)}.

The boundedness and Lipschitz estimates above allow us to apply
Lemma~\ref{lem:levy-girsanov-kl}, which gives \(P_j\sim P_0\) and hence
\emph{(iii)}.

It remains to prove \emph{(iv)}. By
Lemma~\ref{lem:levy-occupation-bound} and the packing property with
\(T=\Gamma\),
\[
\begin{aligned}
&E_{j,n}\left[
    \int_0^{n\Delta}
    \abs{\widetilde f_{(j)}(X_u)-\widetilde f_{(0)}(X_u)}^2\,du
\right] \\
&\qquad\le
4R^2t_0
+
\Gamma n\Delta
\norm{\widetilde f_{(j)}-\widetilde f_{(0)}
}_{L^2(\bbR^d)}^2 \\
&\qquad\le
\frac{\log2}{9}
+
\frac{\log M}{9}
\le
\frac{2\log M}{9},
\end{aligned}
\]
where we used \(M\ge2\).
Hence, by Lemma~\ref{lem:levy-girsanov-kl},
\[
    \int
    \log\left(\frac{dP_j}{dP_0}\right)dP_j
    \le
    \frac{\log M}{9},
    \qquad j=1,\ldots,M.
\]
Averaging over \(j=1,\ldots,M\) proves \emph{(iv)}.
\end{proof}

\begin{proof}[Proof of Theorem~\ref{thm:minimax-lower-gaussian-cp}]
By Lemma~\ref{lem:levy-minimax-risk-reduction}, it is enough to prove the
corresponding lower bound for the \(L^2([0,1]^d)\)-estimation error.

Since every estimator based on the discrete observations is also admissible
under continuous observation,
\begin{equation}
\label{eq:levy-discrete-continuous-comparison}
\begin{aligned}
&\inf_{\widehat f_n^{\rm disc}}
  \sup_{b\in\ClassB^J}
  E_{b,n}\!\left[
      \norm{\widehat f_n^{\rm disc}-f_0}_2^2
  \right] \\
&\qquad\ge
\inf_{\widehat f_n^{\rm cont}}
  \sup_{b\in\ClassB^J}
  E_{b,n}\!\left[
      \norm{\widehat f_n^{\rm cont}-f_0}_2^2
  \right],
\end{aligned}
\end{equation}

By Lemma~\ref{lem:levy-fano-family}, for all sufficiently large \(n\) there
exist \(b_0,\ldots,b_M\in\ClassB^J\) satisfying
\eqref{eq:levy-fano-separation}, \eqref{eq:levy-fano-ac}, and
\eqref{eq:levy-fano-kl}. Therefore, Fano's method
\cite[Theorem~2.7]{tsy09} yields
\[
    \liminf_{n\to\infty}
    \phi_n^{-1}
    \inf_{\widehat f_n^{\rm cont}}
    \sup_{b\in\ClassB^J}
    E_{b,n}\left[
        \norm{\widehat f_n^{\rm cont}-f_0}_2^2
    \right]
    >0 .
\]
Together with \eqref{eq:levy-discrete-continuous-comparison}, this gives
\[
    \liminf_{n\to\infty}
    \phi_n^{-1}
    \inf_{\widehat f_n}
    \sup_{b\in\ClassB^J}
    \bbE\left[
        \norm{\widehat f_n-f_0}_2^2
    \right]
    >0 ,
\]
where the infimum is now over estimators based on the discrete observations.
Finally, Lemma~\ref{lem:levy-minimax-risk-reduction} gives
\[
    \liminf_{n\to\infty}
    \phi_n^{-1}
    \inf_{\widehat f_n}
    \sup_{b\in\ClassB^J}
    \Risk_n(\widehat f_n,f_0)
    >0 .
\]
This proves the theorem.
\end{proof}

\appendix
\renewcommand{\thetheorem}{\Alph{section}.\arabic{theorem}}

\section{Auxiliary \texorpdfstring{\(\psi_1\)}{psi1}-chaining bound}
\label{app:psi1-chaining}
We next isolate the chaining step used to pass from fixed-pair
\(\psi_1\)-increment bounds to a uniform local estimate.
For completeness, recall that \(\psi_1(x):=e^x-1\) for \(x\ge0\),
and that, for a real-valued random variable \(Y\), the associated Orlicz
norm is defined by
\[
    \norm{Y}_{\psi_1}
    :=
    \inf\left\{
        c>0:
        \bbE\left[
            \psi_1\left(\frac{\abs{Y}}{c}\right)
        \right]
        \le1
    \right\},
\]
with the convention that the infimum over the empty set is \(+\infty\).

The following result is a direct specialization of the maximal
inequality of van der Vaart and Wellner~\cite{van-wel96}.
We say that a stochastic process
\(\{X_t:t\in\mathcal T\}\) indexed by a semimetric space
\((\mathcal T,d)\) is separable if there exists a countable subset
\(\mathcal T_0\subset\mathcal T\) such that, for every \(a>0\),
\[
    \sup_{d(s,t)<a}\abs{X_s-X_t}
    =
    \sup_{\substack{s,t\in\mathcal T_0\\ d(s,t)<a}}
    \abs{X_s-X_t}
    \qquad\text{a.s.}
\]

\begin{lemma}\label{lem:psi1-chaining}
Let \((\mathcal T,d)\) be a semimetric space, and let
\(\{X_t:t\in\mathcal T\}\) be a separable stochastic process. Assume that
for some \(C>0\),
\[
    \norm{X_s-X_t}_{\psi_1}\le C\,d(s,t),
    \qquad s,t\in\mathcal T .
\]
Assume moreover that \(\mathcal T\) is not a singleton. Then, for every
\(\delta\) satisfying
\[
    0<\delta<\operatorname{diam}(\mathcal T,d),
    \qquad
    \operatorname{diam}(\mathcal T,d)
    :=
    \sup_{s,t\in\mathcal T}d(s,t),
\]
there exists a constant \(K>0\), depending only on \(C\), such that
\[
    \bbE\Bigl[
        \sup_{d(s,t)\le\delta}\abs{X_s-X_t}
    \Bigr]
    \le
    K\int_0^\delta
    \log \NN(\varepsilon,\mathcal T,d)\,d\varepsilon .
\]
\end{lemma}
\begin{proof}[Proof of Lemma~\ref{lem:psi1-chaining}]
We prove a slightly stronger Orlicz-norm estimate and then deduce the
claimed expectation bound. For \(\varepsilon>0\), let
\(D(\varepsilon,d)\) denote the packing number of \((\mathcal T,d)\),
that is, the maximal cardinality of an \(\varepsilon\)-separated subset
of \(\mathcal T\), where \(\varepsilon\)-separated means that
\(d(s,t)>\varepsilon\) for every two distinct points \(s,t\). Recall the
standard comparison between covering and packing numbers:
\[
\NN(\varepsilon,\mathcal T,d)\le D(\varepsilon,d)
\le \NN(\varepsilon/2,\mathcal T,d),
\qquad \varepsilon>0 .
\]
Apply the maximal inequality of van der Vaart and Wellner
\cite[Theorem~2.2.4]{van-wel96} with \(\psi=\psi_1\) and with
the auxiliary radius in that theorem chosen as \(\delta\). Then there
exists a constant \(K_0>0\), depending only on \(C\), such that
\[
\norm{
\sup_{d(s,t)\le\delta}\abs{X_s-X_t}
}_{\psi_1}
\le
K_0\left[
\int_0^\delta \psi_1^{-1}\!\bigl(D(\varepsilon,d)\bigr)\,d\varepsilon
+
\delta\,\psi_1^{-1}\!\bigl(D(\delta,d)^2\bigr)
\right].
\]
Since \(\psi_1^{-1}(u)=\log(1+u)\), this gives
\[
\norm{
\sup_{d(s,t)\le\delta}\abs{X_s-X_t}
}_{\psi_1}
\le
K_0\left[
\int_0^\delta \log\!\bigl(1+D(\varepsilon,d)\bigr)\,d\varepsilon
+
\delta\log\!\bigl(1+D(\delta,d)^2\bigr)
\right].
\]
Because \(\mathcal T\) is not a singleton and
\(\delta<\operatorname{diam}(\mathcal T,d)\), there exist
\(s^*,t^*\in\mathcal T\) such that \(d(s^*,t^*)>\delta\), so
\(\{s^*,t^*\}\) is \(\delta\)-separated and hence
\(D(\delta,d)\ge2\). Since \(D(\varepsilon,d)\) is nonincreasing in
\(\varepsilon\), we also have \(D(\varepsilon,d)\ge2\) for every
\(0<\varepsilon\le\delta\). Therefore, for \(m\ge2\),
\[
\log(1+m)\le 2\log m,
\qquad
\log(1+m^2)\le 3\log m .
\]
Consequently,
\[
\norm{
\sup_{d(s,t)\le\delta}\abs{X_s-X_t}
}_{\psi_1}
\le
K_1\left[
\int_0^\delta \log D(\varepsilon,d)\,d\varepsilon
+
\delta\log D(\delta,d)
\right]
\]
for a constant \(K_1>0\) depending only on \(C\). Since
\(\log D(\varepsilon,d)\) is nonincreasing in \(\varepsilon\),
\[
\delta\log D(\delta,d)
\le
\int_0^\delta \log D(\varepsilon,d)\,d\varepsilon .
\]
Therefore, for some constant \(K_2>0\) depending only on \(C\),
\[
\norm{
\sup_{d(s,t)\le\delta}\abs{X_s-X_t}
}_{\psi_1}
\le
K_2
\int_0^\delta \log D(\varepsilon,d)\,d\varepsilon .
\]
Using \(D(\varepsilon,d)\le \NN(\varepsilon/2,\mathcal T,d)\), we obtain
\[
\begin{aligned}
\norm{
\sup_{d(s,t)\le\delta}\abs{X_s-X_t}
}_{\psi_1}
&\le
K_2\int_0^\delta
\log \NN(\varepsilon/2,\mathcal T,d)\,d\varepsilon  \\
&=
2K_2\int_0^{\delta/2}
\log \NN(u,\mathcal T,d)\,du  \\
&\le
2K_2\int_0^\delta
\log \NN(u,\mathcal T,d)\,du .
\end{aligned}
\]

Put
\[
Y:=\sup_{d(s,t)\le\delta}\abs{X_s-X_t}.
\]
By the standard Orlicz-norm comparison
\cite[Problems~2.2.4--2.2.5]{van-wel96}, applied with the fixed
Young function \(\psi_1(x)=e^x-1\), there exists a numerical constant
\(c_{\psi_1}>0\) such that
\[
    \bbE[Y]=\norm{Y}_1\le c_{\psi_1}\norm{Y}_{\psi_1}.
\]
Absorbing the fixed constant \(c_{\psi_1}\) into \(K\), we obtain the
claimed expectation bound with \(K\) depending only on \(C\).
\end{proof}

\section{Packing construction for the compositional class}
\label{app:packing-compositional-class}

\begin{proof}[Proof of Lemma~\ref{lem:packing-compositional-class}]
The construction below is adapted from the packing argument in the proof of
Schmidt-Hieber~\cite[Theorem~3]{sch-hie20}, with additional regularity
chosen for the global drift extension used below.

Set \(N:=n\Delta\). By the definition of the rate-determining index \(i^*\),
\(\phi_n=N^{-2\beta^{**}/(2\beta^{**}+t^*)}\) and
\(\beta^{**}=\beta^*B\). Choose
\[
    \mathsf K(x)
    :=
    c\exp\!\left\{-\frac{1}{x(1-x)}\right\}\one_{(0,1)}(x),
\]
where \(c>0\) is small enough that
\(0\le\mathsf K\le1\) and the derivatives and H\"older seminorms
appearing in the \(\beta^*\)-H\"older quantity are bounded by \(1\).
The function \(\mathsf K\) belongs to \(C_c^\infty(\bbR)\) and
is flat at \(0\) and \(1\); hence
\(\mathsf K^B\in C_c^\infty(\bbR)\).

Let \(\rho\ge1\), to be fixed later, and put
\(m:=\lfloor\rho N^{1/(2\beta^{**}+t^*)}\rfloor\) and \(h:=m^{-1}\).
For \(N\) large enough, \(m\ge1\) and \(h\le1\). Define
\[
    U
    :=
    \left\{
        u\in[0,1]^{t^*}:
        u_r\in\{0,h,\ldots,(m-1)h\},
        \ r=1,\ldots,t^*
    \right\},
\]
\[
    \psi_u(z)
    :=
    h^{\beta^*}
    \prod_{r=1}^{t^*}
    \mathsf K\!\left(\frac{z_r-u_r}{h}\right).
\]
Thus \(\abs{U}=m^{t^*}\). For every multi-index \(\alpha\) with
\(\abs{\alpha}_1<\beta^*\),
\(\norminf{\partial^\alpha\psi_u}
 \le C h^{\beta^*-\abs{\alpha}_1}\le C\);
for \(\abs{\alpha}_1=\underline{\beta^*}\), the corresponding
\((\beta^*-\underline{\beta^*})\)-H\"older seminorm is bounded by \(C\).
Hence the defining H\"older quantity of \(\psi_u\) is bounded uniformly
in \(u\) and \(N\).

For \(w=(w_u)_{u\in U}\in\{0,1\}^{\abs{U}}\), set
\[
    \varphi_w:=\sum_{u\in U}w_u\psi_u .
\]
The nonzero regions of distinct bumps are disjoint, and all derivatives
relevant to the H\"older norm vanish at their common cell boundaries.
Consequently,
\[
    \varphi_w
    \in
    \mathcal C_{t^*}^{\beta^*}([0,1]^{t^*},C),
    \qquad
    0\le\varphi_w\le1,
\]
with \(C\) independent of \(N,\rho\), and \(w\).

We now embed \(\varphi_w\) into the compositional structure. Taking
\([0,1]\) as the intermediate coordinate range, define
\[
    g_i(x)
    :=
    \begin{cases}
    (x_1,\ldots,x_{t^*},0,\ldots,0)^\top,
        & i<i^*,\\[1mm]
    \bigl(\varphi_w(x_1,\ldots,x_{t^*}),0,\ldots,0\bigr)^\top,
        & i=i^*,\\[1mm]
    (x_1^{\beta_i\wedge1},0,\ldots,0)^\top,
        & i>i^* .
    \end{cases}
\]
The dimension assumptions, together with \(t_i\le d_i\), imply that \(t^*\le d_j\) for \(j=0,\ldots,i^*\), so the first \(t^*\) coordinates can be carried to the \(i^*\)-th layer.
The carrier components are smooth,
the bottleneck component satisfies the uniform H\"older bound above,
and \(x\mapsto x^{\beta_i\wedge1}\) belongs to
\(\mathcal C_1^{\beta_i}([0,1],C)\). Hence there is a constant
\(K_{\rm comp}>0\), depending only on
\((q,\mathbf d,\mathbf t,\boldsymbol\beta)\), for which all these maps
satisfy the defining conditions of the compositional class.

Writing
\(g_w:=g_q\circ\cdots\circ g_0\), we therefore have
\[
\begin{gathered}
    g_w(x)
    =
    \varphi_w(x_1,\ldots,x_{t^*})^B
    =
    \sum_{u\in U}
    w_u h^{\beta^{**}}
    H\!\left(
        \frac{(x_1,\ldots,x_{t^*})-u}{h}
    \right),\\
    H(z)
    :=\prod_{r=1}^{t^*}\mathsf K(z_r)^B .
\end{gathered}
\]
For later use, define
\[
    F_w(z)
    :=
    \sum_{u\in U}
    w_u h^{\beta^{**}}
    H\!\left(\frac{z-u}{h}\right),
    \qquad z\in\bbR^{t^*}.
\]
Then \(g_w(x)=F_w(x_1,\ldots,x_{t^*})\) for \(x\in[0,1]^d\).
Here we used the nonnegativity and the pairwise disjoint nonzero sets
of the bumps, together with \(\beta^{**}=\beta^*B\). In particular,
\(H\in C_c^\infty(\bbR^{t^*})\) and
\[
    g_w
    \in
    \ClassG
    (q,\mathbf d,\mathbf t,\boldsymbol\beta,K_{\rm comp})
\]
uniformly in \(N,\rho\), and \(w\).

The representation above also gives the regularity needed for the global
extension. The nonzero sets of the final bumps are pairwise disjoint, and a change
of variables gives
\[
    \norminf{F_w}
    \le h^{\beta^{**}}\norminf{H}
    \le \norminf{H},
    \qquad
    \norminf{\nabla F_w}
    \le h^{\beta^{**}-1}\norminf{\nabla H} .
\]
The assumption \(\beta^{**}\ge1\) therefore yields a Lipschitz bound
uniform in \(N,\rho\), and \(w\).

Choose once and for all
\(\chi\in C_c^\infty(\bbR^d;[0,1])\) such that
\(\chi=1\) on a neighborhood of \([0,1]^d\) and
\(\supp(\chi)\subset\{x:\abs{x}\le R_0\}\). Define
\(G_w(x):=F_w(x_1,\ldots,x_{t^*})\) on \(\bbR^d\) and
\(\widetilde f_w:=\chi G_w\), and let
\(f_w:=\widetilde f_w\one_{[0,1]^d}\). Then
\(f_w|_{[0,1]^d}=g_w\), so
\(f_w\in \ClassGb (q,\mathbf d,\mathbf t,\boldsymbol\beta,K_{\rm comp})\).
Moreover,
\(\supp(\widetilde f_w)\subset\{x:\abs{x}\le R_0\}\) and
\[
    \Lip(\widetilde f_w)
    \le
    \norminf{\chi}\Lip(G_w)
    +\Lip(\chi)\norminf{G_w}
    \le C,
    \qquad
    \norminf{\widetilde f_w}\le C .
\]
These bounds are uniform in \(N,\rho\), and \(w\).

We next compute the \(L^2\)-geometry of the family. For
\(w,w'\in\{0,1\}^{\abs{U}}\), let
\(\operatorname{Ham}(w,w')
:=\sum_{u\in U}\one_{\{w_u\ne w'_u\}}\).
The final bumps have disjoint supports, and a change of variables gives
\[
    \norm{F_w-F_{w'}}_{L^2(\bbR^{t^*})}^2
    =
    \operatorname{Ham}(w,w')\,
    h^{2\beta^{**}+t^*}\norm{H}_2^2 .
\]
Since \(g_w-g_{w'}\) depends only on the first \(t^*\) coordinates,
the same identity holds on the unit cube:
\[
    \norm{f_w-f_{w'}}_{L^2([0,1]^d)}^2
    =
    \operatorname{Ham}(w,w')\,
    h^{2\beta^{**}+t^*}\norm{H}_2^2 .
\]

For \(m^{t^*}\) sufficiently large, the Varshamov--Gilbert bound
provides binary vectors
\(w^{(0)},\ldots,w^{(M)}\in\{0,1\}^{\abs{U}}\), with \(M\ge2\), such that
for universal constants \(c_{\rm H},c_{\rm V}>0\),
\[
    \operatorname{Ham}(w^{(j)},w^{(k)})
    \ge c_{\rm H}m^{t^*},
    \qquad
    \log M\ge c_{\rm V}m^{t^*},
    \qquad 0\le j<k\le M .
\]
It follows that
\[
\begin{aligned}
    \norm{f_{w^{(j)}}-f_{w^{(k)}}}_{L^2([0,1]^d)}^2
    &\ge
    c_{\rm H}\norm{H}_2^2 h^{2\beta^{**}} \\
    &\ge
    c_{\rm H}\norm{H}_2^2
    \rho^{-2\beta^{**}}\phi_n ,
\end{aligned}
\]
where the last inequality uses
\(m\le\rho N^{1/(2\beta^{**}+t^*)}\).

It remains to choose \(\rho\) so that the extensions satisfy the upper
packing bound. Since \(G_w-G_{w'}\) is supported in
\([0,1]^{t^*}\) in its active coordinates, Fubini's theorem yields
\[
    \norm{\widetilde f_w-\widetilde f_{w'}}_{L^2(\bbR^d)}^2
    \le
    C_{\rm ext}
    \norm{f_w-f_{w'}}_{L^2([0,1]^d)}^2,
    \qquad
    C_{\rm ext}:=1\vee(2R_0)^{d-t^*}.
\]
Using \(\operatorname{Ham}(w,w')\le m^{t^*}\), we therefore obtain
\[
\begin{aligned}
    N\norm{\widetilde f_w-\widetilde f_{w'}}_{L^2(\bbR^d)}^2
    &\le
    C_{\rm ext}\norm{H}_2^2
    N m^{t^*}h^{2\beta^{**}+t^*}  \\
    &=
    C_{\rm ext}\norm{H}_2^2Nh^{2\beta^{**}} .
\end{aligned}
\]
On the other hand, \(\log M\ge c_{\rm V}m^{t^*}\). Hence it is enough
that
\[
    Nh^{2\beta^{**}+t^*}
    \le
    \frac{c_{\rm V}}{9T C_{\rm ext}\norm{H}_2^2}.
\]
Let \(p:=2\beta^{**}+t^*\). For each \(T>0\), choose
\(\rho_T\ge1\) sufficiently large that
\[
    2^p\rho_T^{-p}
    \le
    \frac{c_{\rm V}}{9T C_{\rm ext}\norm{H}_2^2}.
\]
If \(N\) is large enough that
\(m=\lfloor\rho_TN^{1/p}\rfloor
 \ge \rho_TN^{1/p}/2\),
then \(Nh^p\le2^p\rho_T^{-p}\), and consequently
\[
    N\norm{\widetilde f_{w^{(j)}}-\widetilde f_{w^{(k)}}
    }_{L^2(\bbR^d)}^2
    \le
    \frac{\log M}{9T},
    \qquad 0\le j<k\le M .
\]
With
\(\kappa_T := c_{\rm H}\norm{H}_2^2\rho_T^{-2\beta^{**}}\),
the lower bound above becomes
\[
    \norm{f_{w^{(j)}}-f_{w^{(k)}}}_{L^2([0,1]^d)}^2
    \ge
    \kappa_T\phi_n .
\]

Finally, choose \(N_T\) large enough that the preceding integer-rounding
condition holds and the Varshamov--Gilbert construction is applicable
whenever \(N\ge N_T\). Relabel
\(f_{(j)}:=f_{w^{(j)}}\) and
\(\widetilde f_{(j)}:=\widetilde f_{w^{(j)}}\),
\(j=0,\ldots,M\). Taking \(K_{\rm pack}\) larger than
\(K_{\rm comp}\) and the uniform Lipschitz constant above, and taking
\(C_{\rm pack}\) larger than the uniform supremum bound, gives all the
claims of the lemma. Here \(K_{\rm pack}\) and \(C_{\rm pack}\) depend
only on \((q,\mathbf d,\mathbf t,\boldsymbol\beta)\), whereas
\(\kappa_T\) and \(N_T\) may additionally depend on \(T\).
\end{proof}

\section*{Acknowledgements}
Teppei Ogihara was supported by Japan Society for the Promotion of Science KAKENHI Grant Number JP25K03083.
ChatGPT (OpenAI) was used to assist with manuscript drafting and editing,
\LaTeX{} formatting, literature searches, and mathematical discussions and
cross-checks. The authors reviewed the AI-assisted material, verified the
proofs through their own analysis and detailed discussions, and take full
responsibility for the manuscript.

\bibliographystyle{plainnat}
\bibliography{references_Levy_DNN}

@book {app09,
    AUTHOR = {Applebaum, David},
     TITLE = {L\'{e}vy processes and stochastic calculus},
    SERIES = {Cambridge Studies in Advanced Mathematics},
    VOLUME = {116},
   EDITION = {Second},
 PUBLISHER = {Cambridge University Press, Cambridge},
      YEAR = {2009},
     PAGES = {xxx+460},
      ISBN = {978-0-521-73865-1},
   MRCLASS = {60-02 (60G44 60G51 60H05 60H30 60J35 60J60 60J65)},
  MRNUMBER = {2512800},
MRREVIEWER = {Dora Sele\v{s}i},
       DOI = {10.1017/CBO9780511809781},
       URL = {https://doi.org/10.1017/CBO9780511809781},
}

@article {ban-kru16,
    AUTHOR = {Ba\~nos, David and Kr\"uhner, Paul},
     TITLE = {Optimal density bounds for marginals of {I}t\^o processes},
   JOURNAL = {Commun. Stoch. Anal.},
  FJOURNAL = {Communications on Stochastic Analysis},
    VOLUME = {10},
      YEAR = {2016},
    NUMBER = {2},
     PAGES = {131--150},
       DOI = {10.31390/cosa.10.2.01},
       URL = {https://doi.org/10.31390/cosa.10.2.01},
}

@article {bar98,
    AUTHOR = {Barndorff-Nielsen, Ole E.},
     TITLE = {Processes of normal inverse {G}aussian type},
   JOURNAL = {Finance Stoch.},
  FJOURNAL = {Finance and Stochastics},
    VOLUME = {2},
      YEAR = {1998},
    NUMBER = {1},
     PAGES = {41--68},
       DOI = {10.1007/s007800050032},
       URL = {https://doi.org/10.1007/s007800050032},
}

@article {car-etal02,
    AUTHOR = {Carr, Peter and Geman, H\'{e}lyette and Madan, Dilip B. and Yor, Marc},
     TITLE = {The fine structure of asset returns: an empirical investigation},
   JOURNAL = {J. Bus.},
  FJOURNAL = {The Journal of Business},
    VOLUME = {75},
      YEAR = {2002},
    NUMBER = {2},
     PAGES = {305--332},
       DOI = {10.1086/338705},
       URL = {https://doi.org/10.1086/338705},
}

@article {com-etal07,
    AUTHOR = {Comte, Fabienne and Genon-Catalot, Valentine and Rozenholc, Yves},
     TITLE = {Penalized nonparametric mean square estimation of the coefficients of diffusion processes},
   JOURNAL = {Bernoulli},
  FJOURNAL = {Bernoulli},
    VOLUME = {13},
      YEAR = {2007},
    NUMBER = {2},
     PAGES = {514--543},
       DOI = {10.3150/07-BEJ5173},
       URL = {https://doi.org/10.3150/07-BEJ5173},
}

@article {dal-rei07,
    AUTHOR = {Dalalyan, Arnak S. and Rei{\ss}, Markus},
     TITLE = {Asymptotic statistical equivalence for ergodic diffusions: the multidimensional case},
   JOURNAL = {Probab. Theory Related Fields},
  FJOURNAL = {Probability Theory and Related Fields},
    VOLUME = {137},
      YEAR = {2007},
    NUMBER = {1-2},
     PAGES = {25--47},
       DOI = {10.1007/s00440-006-0502-7},
       URL = {https://doi.org/10.1007/s00440-006-0502-7},
}

@article {dal05,
    AUTHOR = {Dalalyan, Arnak S.},
     TITLE = {Sharp adaptive estimation of the drift function for ergodic diffusions},
   JOURNAL = {Ann. Statist.},
  FJOURNAL = {The Annals of Statistics},
    VOLUME = {33},
      YEAR = {2005},
    NUMBER = {6},
     PAGES = {2507--2528},
       DOI = {10.1214/009053605000000615},
       URL = {https://doi.org/10.1214/009053605000000615},
}

@article {dex25,
    AUTHOR = {Dexheimer, Niklas},
     TITLE = {Adaptive nonparametric drift estimation for multivariate jump diffusions under sup-norm risk},
   JOURNAL = {Stochastic Process. Appl.},
  FJOURNAL = {Stochastic Processes and their Applications},
    VOLUME = {190},
      YEAR = {2025},
     PAGES = {104741},
       DOI = {10.1016/j.spa.2025.104741},
       URL = {https://doi.org/10.1016/j.spa.2025.104741},
}

@article {gob-etal04,
    AUTHOR = {Gobet, Emmanuel and Hoffmann, Marc and Rei{\ss}, Markus},
     TITLE = {Nonparametric estimation of scalar diffusions based on low frequency data},
   JOURNAL = {Ann. Statist.},
  FJOURNAL = {The Annals of Statistics},
    VOLUME = {32},
      YEAR = {2004},
    NUMBER = {5},
     PAGES = {2223--2253},
       DOI = {10.1214/009053604000000797},
       URL = {https://doi.org/10.1214/009053604000000797},
}

@inproceedings {ima-fuk19,
    AUTHOR = {Imaizumi, Masaaki and Fukumizu, Kenji},
     TITLE = {Deep neural networks learn non-smooth functions effectively},
 BOOKTITLE = {Proceedings of the Twenty-Second International Conference on Artificial Intelligence and Statistics},
    SERIES = {Proceedings of Machine Learning Research},
    VOLUME = {89},
      YEAR = {2019},
     PAGES = {869--878},
 PUBLISHER = {PMLR},
       URL = {https://proceedings.mlr.press/v89/imaizumi19a.html},
}

@book {jac-shi03,
    AUTHOR = {Jacod, Jean and Shiryaev, Albert N.},
     TITLE = {Limit theorems for stochastic processes},
    SERIES = {Grundlehren der mathematischen Wissenschaften [Fundamental
              Principles of Mathematical Sciences]},
    VOLUME = {288},
   EDITION = {Second},
 PUBLISHER = {Springer-Verlag, Berlin},
      YEAR = {2003},
     PAGES = {xx+661},
      ISBN = {3-540-43932-3},
   MRCLASS = {60-02 (60F17 60G48 60H05)},
  MRNUMBER = {1943877},
MRREVIEWER = {Dominique L\'{e}pingle},
       DOI = {10.1007/978-3-662-05265-5},
       URL = {https://doi.org/10.1007/978-3-662-05265-5},
}

@article {koh-lan21,
    AUTHOR = {Kohler, Michael and Langer, Sophie},
     TITLE = {On the rate of convergence of fully connected deep neural network regression estimates},
   JOURNAL = {Ann. Statist.},
  FJOURNAL = {The Annals of Statistics},
    VOLUME = {49},
      YEAR = {2021},
    NUMBER = {4},
     PAGES = {2231--2249},
       DOI = {10.1214/20-AOS2034},
       URL = {https://doi.org/10.1214/20-AOS2034},
}

@book {kul18,
    AUTHOR = {Kulik, Alexei},
     TITLE = {Ergodic behavior of {M}arkov processes: with applications to limit theorems},
    SERIES = {De Gruyter Studies in Mathematics},
    VOLUME = {67},
 PUBLISHER = {De Gruyter, Berlin/Boston},
      YEAR = {2018},
       DOI = {10.1515/9783110458930},
       URL = {https://doi.org/10.1515/9783110458930},
}

@article {mas07,
    AUTHOR = {Masuda, Hiroki},
     TITLE = {Ergodicity and exponential {$\beta$}-mixing bounds for
              multidimensional diffusions with jumps},
   JOURNAL = {Stochastic Process. Appl.},
  FJOURNAL = {Stochastic Processes and their Applications},
    VOLUME = {117},
      YEAR = {2007},
    NUMBER = {1},
     PAGES = {35--56},
      ISSN = {0304-4149},
     CODEN = {STOPB7},
   MRCLASS = {60J60},
  MRNUMBER = {2287102 (2008j:60194)},
       DOI = {10.1016/j.spa.2006.04.010},
       URL = {http://dx.doi.org/10.1016/j.spa.2006.04.010},
}

@article {men-etal21,
    AUTHOR = {Menozzi, St\'{e}phane and Pesce, Antonello and Zhang, Xicheng},
     TITLE = {Density and gradient estimates for non degenerate {B}rownian {SDE}s with unbounded measurable drift},
   JOURNAL = {J. Differential Equations},
  FJOURNAL = {Journal of Differential Equations},
    VOLUME = {272},
      YEAR = {2021},
     PAGES = {330--369},
       DOI = {10.1016/j.jde.2020.09.004},
       URL = {https://doi.org/10.1016/j.jde.2020.09.004},
}

@article {mie21,
    AUTHOR = {Mies, Fabian},
     TITLE = {Estimation of state-dependent jump activity and drift for {M}arkovian semimartingales},
   JOURNAL = {J. Statist. Plann. Inference},
  FJOURNAL = {Journal of Statistical Planning and Inference},
    VOLUME = {210},
      YEAR = {2021},
     PAGES = {114--140},
       DOI = {10.1016/j.jspi.2020.04.009},
       URL = {https://doi.org/10.1016/j.jspi.2020.04.009},
}

@article {nic-ray20,
    AUTHOR = {Nickl, Richard and Ray, Kolyan},
     TITLE = {Nonparametric statistical inference for drift vector fields of multi-dimensional diffusions},
   JOURNAL = {Ann. Statist.},
  FJOURNAL = {The Annals of Statistics},
    VOLUME = {48},
      YEAR = {2020},
    NUMBER = {3},
     PAGES = {1383--1408},
       DOI = {10.1214/19-AOS1851},
       URL = {https://doi.org/10.1214/19-AOS1851},
}

@article {oga-koi24,
    AUTHOR = {Oga, Akihiro and Koike, Yuta},
     TITLE = {Drift estimation for a multi-dimensional diffusion process using deep neural networks},
   JOURNAL = {Stochastic Process. Appl.},
  FJOURNAL = {Stochastic Processes and their Applications},
    VOLUME = {170},
      YEAR = {2024},
     PAGES = {104240},
       DOI = {10.1016/j.spa.2023.104240},
       URL = {https://doi.org/10.1016/j.spa.2023.104240},
}

@article {qia-etal03,
    AUTHOR = {Qian, Zhongmin and Russo, Francesco and Zheng, Weian},
     TITLE = {Comparison theorem and estimates for transition probability densities of diffusion processes},
   JOURNAL = {Probab. Theory Related Fields},
  FJOURNAL = {Probability Theory and Related Fields},
    VOLUME = {127},
      YEAR = {2003},
    NUMBER = {3},
     PAGES = {388--406},
       DOI = {10.1007/s00440-003-0291-1},
       URL = {https://doi.org/10.1007/s00440-003-0291-1},
}

@article {ros07,
    AUTHOR = {Rosi\'{n}ski, Jan},
     TITLE = {Tempering stable processes},
   JOURNAL = {Stochastic Process. Appl.},
  FJOURNAL = {Stochastic Processes and their Applications},
    VOLUME = {117},
      YEAR = {2007},
    NUMBER = {6},
     PAGES = {677--707},
       DOI = {10.1016/j.spa.2006.10.003},
       URL = {https://doi.org/10.1016/j.spa.2006.10.003},
}

@article {sch-hie20,
    AUTHOR = {Schmidt-Hieber, Johannes},
     TITLE = {Nonparametric regression using deep neural networks with {ReLU} activation function},
   JOURNAL = {Ann. Statist.},
  FJOURNAL = {The Annals of Statistics},
    VOLUME = {48},
      YEAR = {2020},
    NUMBER = {4},
     PAGES = {1875--1897},
       DOI = {10.1214/19-AOS1875},
       URL = {https://doi.org/10.1214/19-AOS1875},
}

@article {sch14,
    AUTHOR = {Schmisser, \'{E}meline},
     TITLE = {Non-parametric adaptive estimation of the drift for a jump diffusion process},
   JOURNAL = {Stochastic Process. Appl.},
  FJOURNAL = {Stochastic Processes and their Applications},
    VOLUME = {124},
      YEAR = {2014},
    NUMBER = {1},
     PAGES = {883--914},
       DOI = {10.1016/j.spa.2013.09.012},
       URL = {https://doi.org/10.1016/j.spa.2013.09.012},
}

@inproceedings {suz19,
    AUTHOR = {Suzuki, Taiji},
     TITLE = {Adaptivity of deep {ReLU} network for learning in {B}esov and mixed smooth {B}esov spaces: optimal rate and curse of dimensionality},
 BOOKTITLE = {International Conference on Learning Representations},
      YEAR = {2019},
       URL = {https://openreview.net/forum?id=H1ebTsActm},
}

@book {tsy09,
    AUTHOR = {Tsybakov, Alexandre B.},
     TITLE = {Introduction to nonparametric estimation},
    SERIES = {Springer Series in Statistics},
 PUBLISHER = {Springer, New York},
      YEAR = {2009},
     PAGES = {x+214},
       DOI = {10.1007/b13794},
       URL = {https://doi.org/10.1007/b13794},
}

@book {van-wel96,
    AUTHOR = {van der Vaart, Aad W. and Wellner, Jon A.},
     TITLE = {Weak convergence and empirical processes: with applications to statistics},
    SERIES = {Springer Series in Statistics},
 PUBLISHER = {Springer, New York},
      YEAR = {1996},
     PAGES = {xvi+508},
       DOI = {10.1007/978-1-4757-2545-2},
       URL = {https://doi.org/10.1007/978-1-4757-2545-2},
}

@article {fan-etal22,
    AUTHOR = {Fang, Cheng and Lu, Yubin and Gao, Ting and Duan, Jinqiao},
     TITLE = {An end-to-end deep learning approach for extracting stochastic dynamical systems with {$\alpha$}-stable {L}\'{e}vy noise},
   JOURNAL = {Chaos},
  FJOURNAL = {Chaos: An Interdisciplinary Journal of Nonlinear Science},
    VOLUME = {32},
      YEAR = {2022},
    NUMBER = {6},
     PAGES = {063112},
       DOI = {10.1063/5.0089832},
       URL = {https://doi.org/10.1063/5.0089832},
}

@article {ram-sun26,
    AUTHOR = {Ram\'{i}rez-Gonz\'{a}lez, Jos\'{e} Hermenegildo and Sun, Ying},
     TITLE = {Neural networks for the {T}amed--{M}ilstein approximation of {SDE}s with additive symmetric jumps},
   JOURNAL = {Chaos},
  FJOURNAL = {Chaos: An Interdisciplinary Journal of Nonlinear Science},
    VOLUME = {36},
      YEAR = {2026},
    NUMBER = {1},
     PAGES = {013122},
       DOI = {10.1063/5.0289767},
       URL = {https://doi.org/10.1063/5.0289767},
}

@article {sch13,
    AUTHOR = {Schmisser, \'{E}meline},
     TITLE = {Penalized nonparametric drift estimation for a multidimensional diffusion process},
   JOURNAL = {Statistics},
    VOLUME = {47},
      YEAR = {2013},
    NUMBER = {1},
     PAGES = {61--84},
       DOI = {10.1080/02331888.2011.591931},
       URL = {https://doi.org/10.1080/02331888.2011.591931},
}

@incollection{sor91,
  author    = {S{\o}rensen, Michael},
  title     = {Likelihood methods for diffusions with jumps},
  booktitle = {Statistical Inference in Stochastic Processes},
  editor    = {Prabhu, N. U. and Basawa, I. V.},
  pages     = {67--105},
  publisher = {Marcel Dekker},
  address   = {New York},
  year      = {1991}
}
\end{document}